\documentclass[11pt]{article}
\usepackage{amsfonts}
\usepackage{graphics}
\usepackage{indentfirst}
\usepackage{color}
\usepackage{cite}
\usepackage{latexsym}
\usepackage[paper=a4paper, left=1.6cm, right=1.6cm, top=1.8cm, bottom=1.6cm, headheight=5.5pt, footskip=0.8cm, footnotesep=0.8cm, centering, includefoot]{geometry}
\usepackage{amsmath}
\allowdisplaybreaks
\usepackage{amssymb}
\usepackage[colorlinks, linkcolor=red]{hyperref}
\hypersetup{urlcolor=red, citecolor=red}
\usepackage[dvips]{epsfig}
\usepackage{amscd}

\usepackage{amsthm}
\usepackage{mathrsfs}
\usepackage{verbatim}
\newtheorem{theorem}{Theorem}[section]
\newtheorem{remark}{Remark}[section]
\newtheorem{defn}{Definition}[section]
\newtheorem{lemma}{Lemma}[section]
\newtheorem{corollary}{Corollary}[section]
\newtheorem{proposition}{Proposition}[section]

\allowdisplaybreaks

\let\pa=\partial
\let\f=\frac

\def\na{\nabla}

\renewcommand{\div}{{\rm div}}

\DeclareMathOperator{\divv}{div}

\DeclareMathOperator{\curl}{curl}

\makeatletter
\@addtoreset{equation}{section}
\makeatother
\makeatletter
\@addtoreset{equation}{section}
\makeatother

\title{Global well-posedness for the compressible Navier--Stokes equations in 3D bounded domains under a new smallness condition
\thanks{This research was partially supported by National Natural Science Foundation of China (No. 12371227) and Fundamental Research Funds for the Central Universities (No. SWU--KU24001).}
}

\author{Lin Xu,\ Xin Zhong{\thanks{Corresponding author. E-mail addresses: mathxu@email.swu.edu.cn (L. Xu), xzhong1014@amss.ac.cn (X. Zhong).}}
\date{}\\
\footnotesize School of Mathematics and Statistics, Southwest University, Chongqing 400715, P. R. China}
\begin{document}

\maketitle

\begin{abstract}
We investigate the three-dimensional barotropic compressible Navier--Stokes equations involving vacuum in a bounded domain with Navier-slip boundary conditions. It is shown that such a model admits a unique global strong solution provided the gradient of the initial velocity and the initial pressure deviation are suitably small in $L^2$. In particular, \textit{the smallness threshold depends only on the known parameters and the prescribed density bounds.} Our result provides a new analytical framework for the global well-posedness theory of the compressible Navier--Stokes equations.
\end{abstract}

\textit{Key words and phrases}. Compressible Navier--Stokes equations; Navier-slip boundary conditions; vacuum; global strong solutions.

2020 \textit{Mathematics Subject Classification}. 35Q35; 35R05; 76N06; 76N10.

\setcounter{tocdepth}{2}
\tableofcontents

\section{Introduction}

\subsection{Background and main result}

The motion of a viscous compressible barotropic fluid is governed by the compressible Navier--Stokes equations
\begin{align}\label{a1}
\begin{cases}
\rho_t+\operatorname{div}(\rho u)=0,\\
(\rho u)_t+\operatorname{div}(\rho u\otimes u)-\mu\Delta u-(\mu+\lambda)\nabla\operatorname{div}u+\nabla P
=0,
\end{cases}
\end{align}
where $\rho=\rho(x,t)\geq0$ and $u=u(x,t)$ denote the density and velocity, respectively. The pressure is given by
\begin{align*}
P=P(\rho)=a\rho^\gamma,\quad a>0,\quad \gamma>1,
\end{align*}
and the viscosity coefficients satisfy
\begin{align*}
\mu>0,\quad 2\mu+3\lambda\geq0.
\end{align*}
Throughout this paper, we normalize $a=1$ for simplicity. We consider system \eqref{a1} in a simply connected bounded domain $\Omega\subset\mathbb R^3$ with smooth boundary $\partial\Omega$ consisting of a finite number of two-dimensional connected components, subject to the initial data
\begin{align}\label{1.2}
\rho(x,0)=\rho_0(x),\quad \rho u(x,0)=\rho_0u_0(x), \quad x\in\Omega,
\end{align}
and the slip boundary condition
\begin{align}\label{a4}
u\cdot n=0,\quad \operatorname{curl}u\times n=0, \quad x\in\partial\Omega,\ t>0,
\end{align}
where $n$ denotes the unit outward normal vector to $\partial\Omega$.

The global-in-time existence of solutions to the compressible Navier--Stokes equations is one of central topics in the mathematical theory of viscous compressible flows. For the Cauchy problem in the whole space, the pioneering work of Matsumura and Nishida \cite{MN80} established the global existence of smooth solutions near a constant equilibrium state under a small perturbation
assumption in higher-order Sobolev spaces. Later, Danchin \cite{D2000} developed the global well-posedness theory in critical Besov spaces, which was subsequently extended and refined in various critical frameworks; see, for instance, \cite{CharveDanchin2010,CMZ2010,H2011a,H2011b,Danchin2014}. These theories are mainly concerned with perturbations of non-vacuum constant states. When the initial density is allowed to vanish, the analysis
becomes substantially more delicate due to the degeneracy of the momentum
equation on the vacuum set. The compatibility condition on the initial data plays a fundamental role in the strong solution theory with vacuum. For example, the local theory developed by Cho--Choe--Kim \cite{CCK04} requires the compatibility condition
\begin{align}\label{xiang}
-\mu\Delta u_0-(\mu+\lambda)\nabla\operatorname{div}u_0
+\nabla P(\rho_0)=\sqrt{\rho_0}g
\end{align}
for some $g\in L^2$. Such a condition guarantees the required regularity of
the material acceleration at the initial time and has become a standard
assumption in the classical theory involving vacuum. Subsequently, Gong--Li--Liu--Zhang \cite{GLLZ20} and Huang \cite{H21} independently removed the compatibility condition \eqref{xiang} and established local well-posedness results for system \eqref{a1} with vacuum. Meanwhile, Huang--Li--Xin \cite{HLX12} extended the local existence theory in \cite{CCK04} and obtained a unique global classical solution with small initial energy, allowing vacuum and large oscillations. Their approach relies on time-weighted energy estimates, the effective viscous flux, and a Zlotnik-type argument to derive uniform upper bounds for the density. More recently, Wen \cite{W25} introduced a scaling-invariant initial quantity and established global strong solutions of \eqref{a1} in $\mathbb{R}^3$ under a scaling-invariant smallness condition. Further developments in scaling-invariant frameworks can be found in \cite{XZ20262,XZ20261,XZ202611}. Other global existence results for large data or different solution frameworks are available in \cite{HHPZ24,FZZ2018,WWZ14,HHW2019,ZLZ2020}. However, these whole-space theories do not directly address the additional difficulties caused by the presence of physical boundaries.

In contrast to the Cauchy problem in the whole space, the analysis in a
bounded domain is strongly influenced by the interaction between the fluid
and the boundary. Boundary terms arise naturally in higher-order energy
estimates, the Fourier analysis available in the whole space is no longer
directly applicable, and the elliptic estimates for the effective viscous
flux and vorticity have to be adapted to the geometry of the domain.
Moreover, the normal and tangential components of the velocity must be
handled consistently with the prescribed boundary conditions. Consequently,
global results for the Cauchy problem cannot in general be transferred
directly to domains with boundaries. For viscous flows in bounded domains, the choice of boundary conditions plays an essential role in both the physical modeling and the mathematical analysis. Besides the classical no-slip condition $u\big|_{\partial\Omega}=0$, slip boundary conditions provide another important description of the fluid-solid interaction. The classical Navier-slip law can be written in the form
\begin{align}\label{Navi2}
u\cdot n=0,\quad\left(2D(u)n+\vartheta u\right)_{\tau}=0
\quad\text{on }\partial\Omega,
\end{align}
where
\begin{align*}
D(u)=\frac12\left(\nabla u+(\nabla u)^{\operatorname{T}}\right),
\end{align*}
$\vartheta$ denotes the friction coefficient, and $v_\tau$ denotes the
tangential component of a vector $v$ on $\partial\Omega$. \eqref{Navi2}$_{1}$ expresses the impermeability of the boundary and \eqref{Navi2}$_{2}$ relates the tangential velocity to the tangential viscous stress. Such slip-type boundary laws trace back to the pioneering works of Navier \cite{N23} and Maxwell \cite{M79}. Experimental observations have further confirmed the relevance of slip phenomena in viscous flows; see
\cite{ZG01}. From the viewpoint of boundary modeling, Navier-type
conditions naturally arise in flows over rough surfaces and porous interfaces;
see \cite{APV98,JM01,BJ67}. From the mathematical perspective, these boundary
conditions have been extensively studied in the analysis of viscous flows,
including regularity theory and related asymptotic problems; see
\cite{CMR98,CQ10,XX07,XX13} and the references therein. Moreover, Navier-type
boundary conditions can also be derived from kinetic descriptions through the
analysis of the Knudsen layer near a solid boundary; see \cite{JW26,555555}.
It is useful to express the Navier-type condition in terms of the vorticity. Under the impermeability condition, \eqref{Navi2}$_{2}$  can equivalently be written as
\begin{align*}
\operatorname{curl}u\times n=-(Au)_\tau\quad\text{on }\partial\Omega,
\end{align*}
where $A$ is a matrix-valued function on $\partial\Omega$ determined by the friction coefficient and the geometry of the boundary. The boundary condition \eqref{a4}$_2$ considered in the current paper corresponds to the special case $A=0$. This formulation is convenient for deriving estimates for the vorticity and the effective viscous flux.

Global existence for system \eqref{a1} in bounded domains has been studied under various boundary conditions and assumptions on the initial data. Under the no-slip boundary condition, Danchin and Tolksdorf \cite{DT23} established the local well-posedness for general large initial data without vacuum, as well as the global well-posedness for small perturbations of a stable constant equilibrium state. For slip boundary conditions, Kobayashi and Zajaczkowski \cite{KZ99} obtained global solutions close to equilibrium by energy methods. For initial data containing vacuum, Yu and Zhao \cite{YZ17} showed the global existence of classical solutions in a cuboid domain under a suitably small initial-energy condition of the form
\begin{align}\label{con}
E_0\triangleq\int_\Omega\left(\frac12\rho_0|u_0|^2+G(\rho_0)\right)\,dx
\leq\varepsilon (\nabla u_0 ),
\end{align}
where
\begin{align*}
G(\rho)\triangleq\rho\int_{\bar\rho_{0}}^{\rho}
\frac{P(s)-P(\bar\rho_{0})}{s^2}\,ds,\quad \bar\rho_0
\triangleq\frac{1}{|\Omega|}\int_\Omega\rho_0\,dx>0,
\end{align*}
and $\varepsilon (\nabla u_0 )$ stands for an admissible smallness threshold depending on a prescribed bound of the initial velocity gradient.
A key point in their analysis is a piecewise-in-time argument combined with
time-weighted estimates, which provides a uniform-in-time upper bound for the
density. Subsequently, Cai and Li \cite{CL23} investigated general smooth bounded domains with slip boundary conditions under the smallness condition \eqref{con}. Their analysis revealed several new estimates for boundary integrals generated by the slip boundary condition and provided an important framework for the study of global  solutions with vacuum in general bounded domains. It should be noted that the existing global theories in \cite{CL23,YZ17} rely on some admissible smallness threshold $\varepsilon(\nabla u_0)$ depending quantitatively on initial velocity bound. Naturally, one may wonder whether one can obtain global strong solutions in bounded domains with a smallness threshold that is independent of any additional prescribed bound on the initial velocity.

The purpose of the present paper is to give an affirmative answer to this
question. We will establish the global well-posedness of strong solutions to the problem \eqref{a1}--\eqref{a4} under a smallness threshold depending only on the known parameters and the prescribed density bounds, that is, the smallness threshold is completely independent of the initial velocity.

Before stating our main result, we introduce some notation and define the
notion of strong solutions used throughout the paper. We use $\Box$ to
indicate the end of a proof  and $A\triangleq B$ means that $A$ is defined
to be $B$. For two nonnegative quantities $A$ and $B$, we write $A\sim B$ if $cB\leq A\leq CB$ for some positive constants $c$ and $C$. Unless otherwise stated, the constants $c$ and $C$ may vary from line to line and depend only on the fixed parameters of the problem. For a scalar or vector valued function $f$, we denote its material derivative by
\begin{align*}
\dot f\triangleq f_t+u\cdot\nabla f
\end{align*}
and its spatial average by
\begin{align*}
\bar f\triangleq\frac{1}{|\Omega|}\int_\Omega f\,dx.
\end{align*}
For $1\leq p\leq\infty$ and $k\geq1$, we use the standard notation
\begin{align*}
L^p=L^p(\Omega),\ W^{k,p}=W^{k,p}(\Omega),\ H^k=W^{k,2}(\Omega),\ H_b^1=\big\{f \in H^1: (f\cdot n)|_{\partial\Omega}=(\curl f \times n)|_{\partial\Omega}=0\big\}.
\end{align*}

It follows from the continuity equation \eqref{a1}$_1$ and the boundary
condition\eqref{a4} that the total mass is conserved. Hence
\begin{align*}
\bar\rho\triangleq\bar\rho(t)=\bar\rho_0=
\frac{1}{|\Omega|}
\int_\Omega\rho_0\,dx>0,
\qquad\text{for all }t\geq0.
\end{align*}
For later use, we also define
\begin{align*}
\rho_{*}\triangleq \|\rho_{0} \|_{L^{\infty}},\qquad\bar P(t)
=\frac{1}{|\Omega|}\int_\Omega P(\rho(x,t))\,dx,
\quad\bar P_0\triangleq\frac{1}{|\Omega|}\int_\Omega P(\rho_0(x))\,dx.
\end{align*}
The effective viscous flux and vorticity are defined by
\begin{align*}
F\triangleq(2\mu+\lambda)\operatorname{div}u-\bigl(P-\bar P\bigr), \quad \omega\triangleq  \operatorname{curl}u=\na\times u.
\end{align*}
Since $(u\cdot n)\big|_{\partial\Omega}=0$ and
\begin{align}\label{barp}
\int_\Omega (P-\bar P)\,dx=0,
\end{align}
it follows from the definition of $F$ that $\bar F=0$.

We next give the precise definition of strong solutions.
\begin{defn}\label{def}
Let $T>0$ and $q\in(3,6)$. Assume that
\begin{align}\label{sou}
0\leq\rho_0\in W^{1,q},\quad u_0\in H_{b}^1.
\end{align}
A pair $(\rho,u)$ is called a strong solution to \eqref{a1}--\eqref{a4} on $\Omega\times[0,T]$ if
\begin{align*}
\begin{cases}
\rho \in C\big([0,T];L^{2}\big)\cap L^{\infty}\big(0,T; W^{1,q}\big),\
\rho_{t}\in L^{\infty}\big(0,T;L^{2}\big), \\
\rho u\in C\big([0,T];L^{2}\big),\ u\in L^{\infty}\big(0,T;H^{1}\big)
\cap L^{2}\big(0,T;H^{2}\big),\ \sqrt{\rho} u_{t}\in L^{2}\big(0,T;L^{2}\big), \\
\sqrt{t} u\in L^{\infty}\big(0,T;H^{2}\big)\cap L^{2}\big(0,T;W^{2,q}\big),
\ \sqrt{t} u_{t}\in L^{2}\big(0,T;H^{1}\big).
\end{cases}
\end{align*}
Moreover, $(\rho,u)$ satisfies \eqref{a1} almost everywhere in
$\Omega\times(0,T)$. A strong solution is called a global strong solution if it is a strong solution on $\Omega\times[0,T]$ for any $T>0$.
\end{defn}

Our main result can be stated as follows.
\begin{theorem}\label{th1}
Let $q\in(3,6)$ and assume that the initial data $(\rho_0,u_0)$ satisfies \eqref{sou}. Then there exists a positive constant $\varepsilon_0=\varepsilon_0\bigl(\gamma, \mu, \lambda, \bar\rho, \rho_{*}, \Omega\bigr)$ such that if
\begin{align}\label{sca1}
\|\nabla u_0\|_{L^2}^2+\|P(\rho_0)-\bar P_0\|_{L^2}^2
\leq\varepsilon_0,
\end{align}
the problem \eqref{a1}--\eqref{a4} admits a unique global strong solution $(\rho,u)$. Moreover, for any $r\in [1,\infty)$ and $p\in [1,6]$, there exist positive constants $C$ and $\eta$ depending only on $\mu, \lambda, \gamma, \bar{\rho}, \rho_{*}, \Omega, \|\nabla u_0\|_{L^2}, r$, and $p$ such that, for any $t\geq1$,
\begin{align}\label{decay}
\|\rho(\cdot,t)-\bar\rho_{0}\|_{L^r}+\|u(\cdot,t)\|_{W^{1,p}} +\| \sqrt{\rho}\dot{u} (\cdot,t)\|^2_{L^2} \leq Ce^{-\eta t}, \quad t\geq1.
\end{align}
Furthermore, if additionally $\rho_0(x_0)=0$ for some $x_0\in\Omega$, then for any $r_1\in(3,q]$, there exist positive constants $\widetilde C_1$ and $\widetilde C_2$ depending only on $\mu, \lambda, \gamma, \bar\rho, \rho_*, \Omega, \|\nabla u_0\|_{L^2}$, and $r_1$ such that, for any $t\geq1$,
\begin{align}\label{blow}
\|\nabla\rho(\cdot,t)\|_{L^{r_1}}\geq \widetilde C_1e^{\widetilde C_2t}.
\end{align}
\end{theorem}

\begin{remark}
Compared with previous global well-posedness theories of barotropic compressible flows in bounded domains \cite{YZ17,CL23}, the smallness threshold $\varepsilon_0$ in \eqref{sca1} is independent of any additional bound on the initial velocity. Certainly, one may wonder whether we can establish the global well-posedness of \eqref{a1}--\eqref{a4} under a smallness condition analogous to the Cauchy problem in \cite{W25,XZ20261,XZ20262} whose threshold is independent of any initial dada. We leave this problem for future investigation.
\end{remark}

\begin{remark}
We emphasize that the assumption $\rho_0\in W^{1,q}$ guarantees the Lipschitz control of the ow of the velocity field, ensuring the uniqueness of solutions. It is difficult for us to weaken this regularity requirement as the Cauchy problem in \cite{Dan23,LZ}, imposing merely nonnegative bounded initial density. New ideas are required to handle this case, which will be left for future studies.
\end{remark}

\begin{remark}
It should be pointed out that special Navier-slip condition \eqref{a4} is imposed for the velocity in order to simplify the statements and present our ideas clearly. By taking advantage of techniques in \cite{CL23}, the result of Theorem \ref{th1} can be extended to the more general slip boundary condition
\begin{align*}
u\cdot n=0,\quad\operatorname{curl}u\times n=-(Au)_\tau\quad\text{on }\partial\Omega,
\end{align*}
where $A$ is a sufficiently smooth positive semi-definite matrix-valued function on $\partial\Omega$.
\end{remark}

\subsection{Strategy of the proof}
We explain the main difficulties and key ideas in the proof of
Theorem \ref{th1}. One of challenges is to establish global {\it a priori}
estimates without imposing the classical compatibility condition \eqref{xiang} on the initial data. The other difficulty arises from the effect of domain involving physical boundary. Although the velocity field satisfies Navier-slip boundary condition \eqref{a4}, yet its material derivative $\dot u$ does not satisfy such a condition, which leads to additional boundary terms in the higher-order energy estimates. The proof is based on a bootstrap argument; see Proposition \ref{pro}. To this end, we need to close the bootstrap assumptions, which consist of the uniform upper bound of the density and the lower-order estimates for the velocity field; see \eqref{cp1}. The proof is divided into several steps,
and we describe the main ideas below.
		
Under the assumption of an uniform upper bound for the density, we first
establish the relationship between the potential energy and the pressure deviation (see Lemma \ref{Gg}). In particular, by using the conservation of total mass,
\begin{align*}
\int_\Omega (\rho-\bar\rho )\,dx=0,
\end{align*}
we prove that
\begin{align*}
\int_\Omega G(\rho)\,dx\sim\|P-\bar P\|_{L^2}^2,
\end{align*}
which is one of motivations for introducing the quantity $D(t)$ in \eqref{dd}.
Combining the div--curl estimate in Lemma \ref{lm24} with the basic energy
estimate \eqref{w8} yields the lower-order energy estimate \eqref{basic}.  	
Next, by applying the Bogovskii operator $\mathcal{B}[P-\bar P]$ (see Lemma \ref{l22}) to the momentum equation, we derive the space-time estimate for the pressure deviation
\begin{align*}
P-\bar P\in L^2(0,T;L^2).
\end{align*}
Based on these estimates, we construct a suitable Lyapunov functional
$\mathcal{L}(t)$ satisfying
\begin{align*}
\frac{d}{dt}\mathcal{L}(t)+ \mathcal{L}(t)\leq0,
\end{align*}	
where
\begin{align*}
\mathcal{L}(t)\sim\|\sqrt{\rho}u(t)\|_{L^2}^2+
\|\left( P -\bar P\right) (t) \|_{L^2}^2.
\end{align*}
Consequently, we obtain the exponential decay estimate \eqref{w1}.
	
To close the second part of the bootstrap assumption \eqref{cp1}$_{2}$, we establish the estimate of $D(t)$ (see \eqref{xxxx1}). More precisely, to derive an uniform estimate of $\|\nabla u\|_{L^2}$, we test the momentum equation with $\dot u$. Unlike the whole-space case, integration by parts in bounded domains produces nontrivial boundary terms. Thus, we need to
handle (see \eqref{I12} and \eqref{3.42})
\begin{align*}
-\int_{\partial\Omega}\sigma^m(P-\bar P)
(u\cdot\nabla u)\cdot n\,dS \quad\text{	and}\quad  (2\mu+\lambda)
\int_{\partial\Omega}\sigma^m\operatorname{div}u
(u\cdot\nabla u)\cdot n\,dS.
\end{align*}
The main technical difficulty comes from the boundary estimates, especially
the control of the boundary trace of $\nabla u$, which generally requires
information on $\nabla^2u$. However, such an estimate is not
available at this stage. To overcome this difficulty, motivated by \cite{CL23}, we have the identity
\begin{align}\label{bian}
(u\cdot\nabla)u\cdot n=-(u\cdot\nabla)n\cdot u\quad\text{on }\partial\Omega
\end{align}
due to $(u\cdot n)\big|_{\partial\Omega}=0$.
Therefore, the derivative of $u$ on the boundary is transferred to the normal
vector of the boundary, which allows us to avoid the use of unavailable higher-order estimates. This yields the desired estimate \eqref{w30}.
The nonlinear term involving $\|\nabla u\|_{L^2}^6$ in \eqref{lwz5} is controlled by the second bootstrap assumption since it can be absorbed
into the dissipation term once
\begin{align*}
\sup_{0\leq t\leq T}  \|\nabla u(t)\|_{L^2}^2\int_0^T\|\nabla u \|_{L^2}^2\,dt
\end{align*}	
is sufficiently small, which leads to \eqref{xxxx1}. Combining this estimate with the lower-order decay estimate derived, we further obtain the exponential decay of $D(t)$ by constructing a suitable Lyapunov functional (see \eqref{fir}). Here the second part of the bootstrap assumption in \eqref{cp1} is closed; see Lemma \ref{sca} for details.
		
It remains to close the first part of the bootstrap assumption \eqref{cp1}$_{1}$, the uniform upper bound of the density. Along the particle trajectories, the continuity equation can be rewritten as a Zlotnik-type ordinary differential equation (see \eqref{zlo}). The key issue is to establish the estimate of $\|F\|_{L^\infty}$ (see \eqref{b-short-split}), which will be divided into two time intervals. For the
short-time interval, we introduce time-weighted quantities
\begin{gather*}
A_1(T)\triangleq\sup_{0\leq t\leq T}
\left( \sigma \|\nabla u(t)\|_{L^2}^2\right)
+\int_0^T\sigma\|\sqrt\rho\dot u \|_{L^2}^2\,dt,
\\
B_1(T)\triangleq\sup_{0\leq t\leq T}\left( \sigma
\|\sqrt\rho\dot u(t)\|_{L^2}^2\right)+\int_0^T\sigma\|\nabla\dot u \|_{L^2}^2\,dt.
\end{gather*}	
To estimate $B_1(T)$, we apply the operator
$\sigma^m\dot u^j\left(\frac{\partial}{\partial t}
+\operatorname{div}(u\cdot)\right)$
to the $j$-th component of the momentum equation \eqref{a1}$_{2}$. Due to
$(\dot u\cdot n)\big|_{\partial\Omega}\neq0$, an additional boundary term
\begin{align*}
\int_{\partial\Omega}\sigma^mF_t\dot u\cdot n\,dS
\end{align*}	
appears (see \eqref{J1-1}). To treat this term, we introduce
\begin{align*}
u^\perp\triangleq-u\times n,\quad
h\triangleq(\nabla n+(\nabla n)^{\mathsf T})u,
\end{align*}	
which transforms the boundary term into (see \eqref{bo11})
\begin{align*}
\int_{\partial\Omega}\sigma^mFh\cdot\dot u\,dS-\int_{\partial\Omega}
\sigma^mFh\cdot(u\cdot\nabla u)\,dS.
\end{align*}	
Then, using the identity
\begin{align}\label{hi}
u^\perp\times n=(-u\times n)\times n=u\quad\text{on }\partial\Omega,
\end{align}
we obtain the desired boundary estimate \eqref{bo12}.
Combining this estimate with the bound of $A_1(T)$ established in Lemma
\ref{Aclosed}, we can close $B_1(T)$ (see \eqref{mB1}). Consequently, we show
the uniform upper bound of the density on the short-time interval (see
\eqref{density-short-F1}--\eqref{xaio1}). For the large-time interval, the exponential decay estimates \eqref{fir} and \eqref{w41} imply the time integrability of $\|F\|_{L^\infty}$ along particle trajectories. Applying Zlotnik's inequality again gives the uniform density bound \eqref{xiao2}.
		
Finally, once the uniform upper bound of the density is established, we propagate the spatial regularity of the density in Lemma \ref{lem-high-order}.
The $W^{1,q}$-estimate of the density follows from the Lam\'e elliptic estimates, the BKM type inequality \eqref{wuqiong}, and
the time integrability of $\|\rho\dot u\|_{L^q}$ (see \eqref{rho-dotu-L1q}). These estimates together with the local existence theory and the continuation argument imply that the local strong solution can be extended globally.

\textbf{Organization of the paper.} Some important inequalities and auxiliary lemmas will be given in Section \ref{sec2}. Section \ref{se3} is devoted to establishing some lower-order estimates and higher-order estimates are obtained in Section \ref{se4}, while the proof of Theorem \ref{th1} is given in Section \ref{se5}.

\section{Preliminaries}\label{sec2}

In this section we collect some auxiliary results that will be used throughout the proof. We start with the local well-posedness result in \cite[Theorem 1.1]{GLLZ20}.
\begin{proposition}[Local well-posedness]\label{local} Assume that the initial data $(\rho_{0},u_{0})$ satisfies \eqref{sou}. Then there exists a time $T >0$ such that the problem \eqref{a1}--\eqref{a4}  admits a unique strong solution in $\Omega\times(0,T ]$. \end{proposition}

We  record two basic inequalities that will be repeatedly used in deriving such estimates. The first one is the following Gagliardo--Nirenberg  inequality (see, e.g., \cite[Lemma 2.3]{LWZ26}).
\begin{lemma}[Gagliardo--Nirenberg  inequality]\label{l21}
Assume that $\Omega$ is a bounded Lipschitz domain in $\mathbb{R}^3$. For $p\in [2, 6]$, $q\in (1, \infty)$, and $r\in (3, \infty)$, there exist generic constants $C_i>0\ (i\in\{1,2,3,4\})$ which may depend only on $p$, $q$, $r$, and $\Omega$ such that, for any $f\in H^1$ and $g\in L^q\cap W^{1, r}$,
	\begin{gather*}
		\|f\|_{L^p}\le C_1\|f\|_{L^2}^\f{6-p}{2p}\|\na f\|_{L^2}^\f{3p-6}{2p}+C_2\|f\|_{L^2},\\
		\|g\|_{L^\infty}\le C_3\|g\|_{L^q}^{\frac{q(r-3)}{3r+q(r-3)}}\|\na g\|_{L^r}^{\frac{3r}{3r+q(r-3)}}+C_4\|g\|_{L^2}.
	\end{gather*}
	Moreover, if $\int_{\Omega}f(x)\,dx=0$ or $(f\cdot n)|_{\partial\Omega}=0$ or $(f\times n)|_{\partial\Omega}=0$, we can choose $C_2=0$. Similarly, the constant $C_4=0$ provided $\int_{\Omega}g(x)\, dx=0$ or $(g\cdot n)|_{\partial\Omega}=0$ or $(g\times n)|_{\partial\Omega}=0$.
\end{lemma}

The following generalized Poincar\'e inequality (see, e.g., \cite[Lemma 8]{BS2012}) allows one to recover the required control from either the normal or tangential boundary condition.
\begin{lemma}[Poincar\'e inequality]\label{l23}
Let $\Omega\subset\mathbb{R}^3$ be bounded with Lipschitz boundary. Then, for $1<p<\infty$, there exists a positive constant $C$ depending only on $p$ and $\Omega$ such that
\begin{align}\label{z2.5}
\|f\|_{L^p} \leq  C\|\nabla f\|_{L^p},
\end{align}
for each vector field $f\in W^{1,p}$ satisfying either $(f\cdot n)|_{\partial\Omega}=0$ or $(f\times n)|_{\partial\Omega}=0$.
\end{lemma}

Next, we recall the following result of the Bogovskii operator (see, e.g., \cite[Lemma 3.17]{NS04}), which is useful to obtain quantitative control of the pressure deviation $P-\bar P$.
\begin{lemma}\label{l22}
Let $\Omega\subset\mathbb{R}^3$ be a bounded Lipschitz domain. Then, for any $p\in(1,\infty)$, there exists a linear operator $\mathcal{B}=\left[\mathcal{B}_{1}, \mathcal{B}_{2}, \mathcal{B}_{3}\right]: \left\{f \in L^{p}(\Omega): \bar{f}=0\right\}\rightarrow \big(W_{0}^{1, p}\big)^{3}$ such that
\begin{align*}
\begin{cases}
\operatorname{div} \mathcal{B}[f]=f, & x \in \Omega, \\
\mathcal{B}[f]=0, & x \in \partial \Omega,
\end{cases}
\end{align*}
and
\begin{align*}
\| \mathcal{B}[f]\|_{W_{0}^{1, p}} \leq C(p,\Omega)\|f\|_{L^{p}}.
\end{align*}
Moreover, if $f$ can be written in the form $f=\operatorname{div} g$ for a certain $g\in L^p$ satisfying $(g \cdot n)|_{\partial \Omega}=0$, then
\begin{align*}
\|\mathcal{B}[f]\|_{L^{p}} \leq C(p,\Omega)\|g\|_{L^{p}}.
\end{align*}
\end{lemma}

We next recall two div--curl estimates adapted to vector fields satisfying normal or tangential boundary conditions. The following lemma is taken from \cite[Proposition 2.6 and Corollary 2.8]{A2014}.
\begin{lemma}\label{lm24}
Let $1<q<+\infty$, and assume that $\Omega\subset\mathbb R^3$ is a simply connected bounded domain with smooth boundary $\partial\Omega$ having a finite number of two-dimensional connected components. Then, for every $v\in W^{1,q}$, there exists a constant $C>0$ such that
\begin{align}\label{f6}
\|\nabla v\|_{L^q}\le C(\|\operatorname{div} v\|_{L^q}+\|\operatorname{curl} v\|_{L^q})
\end{align}
if $(v\cdot n)\big|_{\partial\Omega}=0$, and
\begin{align}\label{f7}
\|\nabla v\|_{L^q}\leq C(\|\operatorname{div} v\|_{L^{q}}+\|\operatorname{curl} v\|_{L^{q}}+\|v\|_{L^q})
\end{align}
if $(v\times n)\big|_{\partial\Omega}=0$.
\end{lemma}

Next, recall the effective viscous flux  as follows:
\begin{align*}
F=(2\mu+\lambda)\divv u-(P-\bar P).
\end{align*}
For $F$, $\operatorname{curl}u$, and $\nabla u$, we have the following key {\it a priori} estimates.
\begin{lemma}\label{l24}
Let $(\rho, u)$ be a strong solution of \eqref{a1}--\eqref{a4} on $\Omega\times(0, T]$. Then, for any $p\in[2, 6]$, there exists a positive constant $C$ depending only on $p$, $\mu$, $\lambda$, and $\Omega$ such that
\begin{gather}
\|\nabla F\|_{L^p} \le C \|\rho\dot u\|_{L^p},\label{ff8}\\
 \|\nabla\operatorname{curl}u\|_{L^p}\le C\big(\|\rho\dot u\|_{L^p}+\|\nabla u\|_{L^2}\big),\label{f8}\\
\|\nabla u\|_{L^p} \le C\|\rho\dot u\|_{L^2}^{\frac{3p-6}{2p}}\big(\|\nabla u\|_{L^2}+\|P-\bar P\|_{L^2}\big)^{\frac{6-p}{2p}}
+C\big(\|\nabla u\|_{L^2}+\|P-\bar P\|_{L^p}\big).\label{f12}
\end{gather}
\end{lemma}
\begin{proof}
Taking the divergence of \eqref{a1}$_{2}$, one finds that the effective viscous flux $F$ satisfies
\begin{align}
\begin{cases}
\Delta F=\divv(\rho\dot u) &\text{in}~\Omega,\\
\f{\pa F}{\pa n}=\rho\dot u\cdot n &\text{on}~\pa\Omega.
\end{cases}
\end{align}
It follows from \cite[Lemma 4.27]{NS04} that
\begin{align}\label{ff14}
\|\na F\|_{L^p}\le C\|\rho\dot u\|_{L^p}.
\end{align}
On the other hand, we can rewrite $\eqref{a1}_2$ as
\begin{align*}
\mu\na\times \operatorname{curl}u=\na F-\rho\dot u.
\end{align*}
By \eqref{a4}, \eqref{f7}, and \eqref{ff14}, we have
\begin{align}\label{ff16}
\|\nabla \operatorname{curl}u\|_{L^p}\le C(\|\na\times\operatorname{curl}u\|_{L^p}+\|\operatorname{curl}u\|_{L^p})\le C(\|\rho\dot u\|_{L^p}+\|\nabla u\|_{L^p}).
\end{align}

One deduces from $\bar F=0$, \eqref{ff14}, and Gagliardo--Nirenberg inequality that, for $p\in [2, 6]$,
\begin{align}\label{z2.19}
\|F\|_{L^p} \le C\|F\|_{L^2}^\frac{6-p}{2p}\|\nabla F\|_{L^2}^\frac{3p-6}{2p}
 \le C\|\rho\dot{u}\|_{L^2}^\frac{3p-6}{2p}
(\|\nabla u\|_{L^2}+\|P-\bar{P}\|_{L^2})^\frac{6-p}{2p}.
\end{align}
Owing to $(\operatorname{curl}u\times n)\big|_{\partial\Omega}=0$ and \eqref{ff16}, we get that
\begin{align*}
\|\operatorname{curl}u\|_{L^p}&\le C\|\operatorname{curl}u\|_{L^2}^\frac{6-p}{2p}\|\nabla\operatorname{curl}u\|_{L^2}^\frac{3p-6}{2p}\nonumber\\
&\le C\big(\|\rho\dot{u}\|_{L^2}+\|\nabla u\|_{L^2}\big)^\frac{3p-6}{2p}\|\nabla u\|_{L^2}^\frac{6-p}{2p}\nonumber\\
&\le C\|\rho\dot{u}\|_{L^2}^\frac{3p-6}{2p}\|\nabla u\|_{L^2}^\frac{6-p}{2p}
+C\|\nabla u\|_{L^2},
\end{align*}
which along with \eqref{f6} and \eqref{z2.19} yields that
\begin{align*}
\|\nabla u\|_{L^p}&\le C\big(\|{\rm div}\,u\|_{L^p}+\|\operatorname{curl}u\|_{L^p}\big)\\[2mm]
&\le C\big(\|F\|_{L^p}+\|P-\bar{P}\|_{L^p}+\|\operatorname{curl}u\|_{L^p}\big)\\
&\le C\|\rho\dot{u}\|_{L^2}^\frac{3p-6}{2p}\big(\|\nabla u\|_{L^2}+\|P-\bar{P}\|_{L^2}\big)^\frac{6-p}{2p}+C\big(\|\nabla u\|_{L^2}+\|P-\bar{P}\|_{L^p}\big),
\end{align*}
proving \eqref{f12}.
\end{proof}

Under the slip boundary condition \eqref{a4}, $\dot u$ does not itself satisfy a homogeneous boundary condition. A suitable correction of $\dot u$ is needed in order to apply div--curl and Poincar\'e-type estimates. Using the identity
$u^\perp\times n=u$ on $\partial\Omega$, we obtain
\begin{align*}
\bigl(\dot u+(u\cdot\nabla n)\times u^\perp\bigr)\cdot n=0
\quad\text{on }\partial\Omega.
\end{align*}
Based on this boundary identity, one obtains the following estimates for $\dot u$ (see, e.g., \cite[Lemma 2.10]{CL23}).
\begin{lemma}\label{xu}
Under the assumption of Lemma \ref{lm24}, there exists a positive constant $C$ depending only on $\Omega$ such that
\begin{gather}
\|\dot u\|_{L^6}\le C\big(\|\nabla\dot u\|_{L^2}+\|\nabla u\|_{L^2}^2\big),\label{d11}\\[2mm]
\|\nabla\dot u\|_{L^2}\le C\big(\|\divv\dot u\|_{L^2}+\|\curl\dot u\|_{L^2}+\|\nabla u\|_{L^4}^2\big).\label{d12}
\end{gather}
\end{lemma}

Next, the following Zlotnik's inequality (see \cite[Lemma 1.3]{Z2000}) will be applied to derive an uniform upper bound of the density.
\begin{lemma}\label{zo}
Let the function $y$ satisfy
\begin{align*}
y^{\prime}(t)=g(y)+b^{\prime}(t) \text { on }[0, T],\quad  y(0)=y^{0},
\end{align*}
with $g\in C(\mathbb R)$  and  $y$, $b \in W^{1,1}(0, T)$. If  $g(\infty)=-\infty$  and
\begin{align*}
b\left(t_{2}\right)-b\left(t_{1}\right) \leq N_{0}+N_{1}\left(t_{2}-t_{1}\right),
\end{align*}
for all $0 \leq t_{1}<t_{2} \leq T$  with some  $N_{0} \geq 0$  and  $N_{1} \geq 0$, then
\begin{align*}
y(t) \leq \max \left\{y^{0}, \bar{\zeta}\right\}+N_{0}<\infty\quad \text { on }\quad [0, T],
\end{align*}
where $\bar{\zeta}$ is a constant such that
\begin{align*}
g(\zeta) \leq-N_{1} \quad\text { for }\quad \zeta \geq \bar{\zeta}.
\end{align*}
\end{lemma}

Finally, we use the following BKM type logarithmic inequality in \cite[Lemma 2.7]{CL23} to control $\|\nabla u\|_{L^\infty}$, which plays a crucial role in propagating the $W^{1,q}$-regularity of the density.
\begin{lemma}\label{lem-bkm}
For $3<q<6$, assume that $ u\in W^{2,q}$ and satisfies boundary condition \eqref{a4}. Then there is a constant $C=C(q,\Omega)$ such that
\begin{align}\label{wuqiong}
\|\nabla u\|_{L^\infty}\le C\left(\|\divv u\|_{L^\infty}+\|\operatorname{curl}u\|_{L^\infty} \right)\ln(e+\|\nabla^2u\|_{L^q})+C\|\nabla u\|_{L^2} +C.
\end{align}
\end{lemma}

\section{A \textit{priori} estimates (I): lower-order estimates}\label{se3}

This section is devoted to establishing the lower-order a \textit{priori} estimates which are independent of $T$. According to the local well-posedness
result in Proposition \ref{local}, we fix $T>0$ and denote by
$(\rho,u)$ a strong solution to \eqref{a1}--\eqref{a4} on
$\Omega\times(0,T]$. Throughout this section, $C$ denotes a generic
positive constant depending only on $\mu, \lambda, \gamma, \Omega, \bar\rho$, and $\rho_*$, which may vary from line to line.

For convenience, we define
\begin{align}\label{dd}
	D(t)\triangleq
	\|\nabla u(t)\|_{L^2}^2+
	\|P(\rho(t))-\bar P(t)\|_{L^2}^2.
\end{align}
The following bootstrap proposition is crucial for our arguments.
\begin{proposition}\label{pro}
Assume that the hypotheses of Theorem \ref{th1} hold. Then there exist positive constants $c_0$ and $\varepsilon_1<1/2$, depending only on $\mu, \lambda, \gamma, \Omega, \bar\rho$, and $\rho_{*}$, such that the following statement holds.
	
If
\begin{align}\label{cp1}
\sup_{0\leq t\leq T}\|\rho(t)\|_{L^\infty}\leq2\rho_{*},\quad
\sup_{0\leq t\leq T}D(t)+\sup_{0\leq t\leq T}\|\nabla u(t)\|_{L^2}^2
\int_0^T\|\nabla u \|_{L^2}^2\,dt\leq2\varepsilon
\end{align}
for some $\varepsilon$ satisfying
\begin{align*}
c_{0}\varepsilon_{0}\leq \varepsilon\leq \varepsilon_{1},
\end{align*}
then
\begin{align}\label{cp2}
\sup_{0\leq t\leq T}\|\rho(t)\|_{L^\infty}
\leq\frac74\rho_{*},\quad\sup_{0\leq t\leq T}D(t)
+\sup_{0\leq t\leq T}\|\nabla u(t)\|_{L^2}^2
\int_0^T\|\nabla u\|_{L^2}^2\,dt
\leq\frac74\varepsilon.
\end{align}
\end{proposition}
\begin{proof}
Proposition \ref{pro} follows from Lemmas \ref{sca} and
\ref{lemupper}.
\end{proof}

We start with several equivalent estimates for the relative potential energy.
\begin{lemma}\label{Gg}
Assume that $0\leq\rho\leq2\rho_{*}$. Then there exists a positive constant $C$ such that
\begin{gather}\label{pressure}
\|P-\bar P\|_{L^2}^2\leq\|P-P(\bar\rho)\|_{L^2}^2
\leq C\int_\Omega G(\rho)\,dx,
\\ \label{deng}
\int_\Omega G(\rho)\,dx\sim\|\rho-\bar\rho\|_{L^2}^2
\sim\|P-\bar P\|_{L^2}^2\sim\int_\Omega\bigl(P -P(\bar\rho)\bigr)
(\rho-\bar\rho)\,dx.
\end{gather}
\end{lemma}
\begin{proof}
Owing to $P(\rho)=\rho^\gamma$, a direct calculation gives
\begin{align*}
G(\rho)=\frac{1}{\gamma-1}\left[
\rho^\gamma-\bar\rho^\gamma-\gamma\bar\rho^{\gamma-1}(\rho-\bar\rho)\right].
\end{align*}
Since the function $r\mapsto r^\gamma$ is strictly convex for $\gamma>1$, one has that
\begin{align*}
G(\rho)>0\quad\text{for }\rho\neq\bar\rho.
\end{align*}
For $\rho\neq\bar\rho$, define
\begin{align*}
H(\rho)\triangleq\frac{G(\rho)}{(\rho-\bar\rho)^2}.
\end{align*}
Notice that
\begin{align*}
G(\bar\rho)=G'(\bar\rho)=0,\quad
G''(\bar\rho)=\gamma\bar\rho^{\gamma-2}>0,
\end{align*}
Taylor's formula yields
\begin{align*}
\lim_{\rho\to\bar\rho}H(\rho)=\frac12G''(\bar\rho)
=\frac{\gamma}{2}\bar\rho^{\gamma-2}>0.
\end{align*}
Hence $H$ can be continuously extended to be a positive function on $[0, 2\rho_*]$. Thus,
\begin{align}\label{G-rho-equiv}
\int_\Omega G(\rho)\,dx\sim\|\rho-\bar\rho\|_{L^2}^2.
\end{align}
	
Similarly, the function
\begin{align*}
K(\rho)\triangleq\frac{\bigl(P -P(\bar\rho)\bigr)(\rho-\bar\rho)}
{G(\rho)},\quad \rho\neq\bar\rho,
\end{align*}
is positive. Moreover,
\begin{align*}
\lim_{\rho\to\bar\rho}K(\rho)=\frac{P'(\bar\rho)}{\frac12G''(\bar\rho)}
=2\bar\rho>0.
\end{align*}
So $K$ can also be extended to be a positive continuous function on
$[0,2\rho_{*}]$, and consequently
\begin{align}\label{Gp}
\int_\Omega G(\rho)\,dx\sim\int_\Omega
\bigl(P -P(\bar\rho)\bigr)(\rho-\bar\rho)\,dx.
\end{align}

We next compare $P-\bar P$ with $P-P(\bar\rho)$. For every constant
$c\in\mathbb R$, it follows from \eqref{barp} that
\begin{align*}
\|P-c\|_{L^2}^2=\|P-\bar P\|_{L^2}^2+|\Omega|\,|\bar P-c|^2.
\end{align*}
Taking $c=P(\bar\rho)$ yields
\begin{align}\label{m22}
\|P-\bar P\|_{L^2}^2\leq\|P-P(\bar\rho)\|_{L^2}^2.
\end{align}
Since $0\leq\rho\leq2\rho_{*}$, the mean value theorem gives
\begin{align*}
|P -P(\bar\rho)|\leq C|\rho-\bar\rho|,
\end{align*}
which combined with \eqref{G-rho-equiv} indicates that
\begin{align*}
\|P-P(\bar\rho)\|_{L^2}^2\leq C\|\rho-\bar\rho\|_{L^2}^2
\leq C\int_\Omega G(\rho)\,dx.
\end{align*}
This along with \eqref{m22} implies \eqref{pressure}.

It remains to establish the equivalence between
$\|P-\bar P\|_{L^2}$ and $\|\rho-\bar\rho\|_{L^2}$.
Since
\begin{align*}
\int_\Omega(\rho-\bar\rho)\,dx=0,
\end{align*}
we have
\begin{align*}
	\int_\Omega
	(P -P(\bar\rho))(\rho-\bar\rho)\,dx
	&=
	\int_\Omega
	(P -\bar P)(\rho-\bar\rho)\,dx .
\end{align*}
Therefore, by Cauchy--Schwarz inequality and \eqref{Gp},
\begin{align*}
\|\rho-\bar\rho\|_{L^2}^2
\leq
C
\|P-\bar P\|_{L^2}
\|\rho-\bar\rho\|_{L^2},
\end{align*}
which together with \eqref{pressure} and
\eqref{G-rho-equiv} leads to
\begin{align*}
\int_\Omega G(\rho)\,dx
\sim
\|\rho-\bar\rho\|_{L^2}^2
\sim
\|P-\bar P\|_{L^2}^2. \tag*{\qedhere}
\end{align*}
\end{proof}

\begin{lemma}\label{lem-basic}
Assume that \eqref{cp1} holds. Then
\begin{align}\label{basic1}
\sup_{0\leq t\leq T}
		\left(
		\|\sqrt{\rho}u(t)\|_{L^2}^2
		+
		\|(P-\bar P)(t)\|_{L^2}^2
		\right)
		+
		\int_0^T
		\left(
		\|\nabla u \|_{L^2}^2
		+
		\| P-\bar P \|_{L^2}^2
		\right)\,dt
		\leq
		C D(0).
	\end{align}
	Moreover, there exists a constant $\eta_1>0$ such that
	\begin{align}\label{w1}
	\sup_{0\leq t\leq T}
		\left[ e^{\eta_1t}
		\left(
		\|\sqrt{\rho}u(t)\|_{L^2}^2
		+
		\|(P-\bar P)(t)\|_{L^2}^2
		\right)\right]
		+
		\int_0^T
		e^{\eta_1t}
		\left(
		\|\nabla u \|_{L^2}^2
		+
		\| P-\bar P \|_{L^2}^2
		\right)\,dt
		\leq
		C D(0).
	\end{align}
\end{lemma}

\begin{proof}
{\bf Step 1. Basic energy estimate.}	
Differentiating $G(\rho)$ with respect to $\rho$, we find
	\begin{align*}
	G'(\rho)
	=
	\int_{\bar\rho}^{\rho}
	\frac{P(s)-P(\bar\rho)}{s^2}\,ds
	+
	\frac{P(\rho)-P(\bar\rho)}{\rho},
\end{align*}
	and hence
	\begin{align}\label{Gs}
		\rho G'(\rho)
		=
		G(\rho)+P -P(\bar\rho).
	\end{align}
	Multiplying the continuity equation \eqref{a1}$_1$ by $G'(\rho)$ and
	using \eqref{Gs}, we derive
	\begin{align*}
	G(\rho)_t
	+
	\operatorname{div}\bigl(G(\rho)u\bigr)
	+
	\bigl(P -P(\bar\rho)\bigr)\operatorname{div}u
	=0.
\end{align*}
	Integrating the above identity over $\Omega$ and using the boundary condition $(u\cdot n)\big|_{\partial\Omega}=0$
	gives
	\begin{align}\label{G2}
		\frac{d}{dt}
		\int_\Omega G(\rho)\,dx
		+
		\int_\Omega P \operatorname{div}u\,dx
		=0.
	\end{align}
Multiplying the momentum equation \eqref{a1}$_2$ by
	$u$ and integrating the resultant over $\Omega$, we obtain
	\begin{align*}
		\frac{d}{dt}
		\int_\Omega\frac12\rho|u|^2\,dx
		+
		(2\mu+\lambda)\|\operatorname{div}u\|_{L^2}^2
		+
		\mu\|\operatorname{curl}u\|_{L^2}^2
		-
		\int_\Omega P\operatorname{div}u\,dx
		=0,
	\end{align*}
which together with \eqref{G2} yields
	\begin{align}\label{w8}
		\frac{d}{dt}\int_\Omega
		\left(
		\frac12\rho|u|^2+G(\rho)
		\right)\,dx
		+
		(2\mu+\lambda)\|\operatorname{div}u\|_{L^2}^2
		+
		\mu\|\operatorname{curl}u\|_{L^2}^2
		=0.
	\end{align}
Therefore, integrating \eqref{w8} over $[0,T]$ and using \eqref{f6}, Lemma \ref{Gg}, and Poincar\'e's inequality, one has
	\begin{align}\label{basic}
	\sup_{0\leq t\leq T}
	\left(
	\|\sqrt{\rho}u(t)\|_{L^2}^2
	+
	\|(P-\bar P)(t)\|_{L^2}^2
	\right)
	+
	\int_0^T
	\|\nabla u \|_{L^2}^2\,dt
	\leq
	C D(0).
	\end{align}
	
	{\bf Step 2. Time-integrability of the pressure deviation.}
	Let
$\Phi
	\triangleq
	\mathcal B[P-\bar P],
$
where $\mathcal B$ denotes the Bogovskii operator introduced in Lemma \ref{l22}. Testing \eqref{a1}$_2$ by $\Phi$ gives
	\begin{align*}
		\|P-\bar P\|_{L^2}^2
		= &
		\frac{d}{dt}
		\int_\Omega\rho u\cdot\Phi\,dx
		-
		\int_\Omega
		\rho u\otimes u:\nabla\Phi\,dx
		-
		\int_\Omega
		\rho u\cdot\mathcal B[P_t-\bar P_t]\,dx\nonumber\\
		&
		+
		\mu\int_\Omega\nabla u:\nabla\Phi\,dx
		 +
		(\lambda+\mu)
		\int_\Omega
		(P-\bar P)\operatorname{div}u\,dx\nonumber\\
		\leq& \frac{d}{dt}
		\int_\Omega\rho u\cdot\Phi\,dx+C\|\nabla u\|_{L^2}^2
		\|P-\bar P\|_{L^{3/2}}+\frac12\|P-\bar P\|_{L^2}^2
		+
		C\|\nabla u\|_{L^2}^2\nonumber\\
	    \leq&
		\frac{d}{dt}
		\int_\Omega\rho u\cdot\Phi\,dx
		+
		\frac12\|P-\bar P\|_{L^2}^2
		+
		C\|\nabla u\|_{L^2}^2,
	\end{align*}
due to Lemma \ref{l22}, Poincar\'e's inequality, and the following facts
\begin{gather*}
	\|\mathcal B[P_t-\bar P_t]\|_{L^2}
	\leq
	C\|\nabla u\|_{L^2},\\
	P_t+\operatorname{div}(Pu)
	+
	(\gamma-1)P\operatorname{div}u
	=0,\\ 
		P_t-\bar P_t
		=
		-\operatorname{div}(Pu)
		-
		(\gamma-1)
		\left(
		P\operatorname{div}u
		-
		\overline{P\operatorname{div}u}
		\right).
	\end{gather*}
	Consequently,
	\begin{align}\label{p-time}
		\frac12\|P-\bar P\|_{L^2}^2
		\leq
		\frac{d}{dt}
		\int_\Omega
		\rho u\cdot\mathcal B[P-\bar P]\,dx
		+
		C\|\nabla u\|_{L^2}^2.
	\end{align}
	Integrating \eqref{p-time} over $[0,T]$ and using
	\eqref{basic}, we get
		\begin{align*}
	\int_0^T
	\|  P-\bar P \|_{L^2}^2\,dt
	\leq
	CD(0),
\end{align*}
which combined with \eqref{basic}
yields \eqref{basic1}.
	
	{\bf Step 3. Exponentially weighted estimate.}
	Define
	\begin{align*}
E(t)\triangleq\int_{\Omega}
		\left(
		\frac12\rho|u|^2+G(\rho)
		\right)(x,t)\,dx,\ \ \mathcal C_1(t)
		\triangleq
		C_1E(t)
		-
		\int_\Omega
		\rho u\cdot\mathcal B[P-\bar P]\,dx,
	\end{align*}
	where $C_1>0$ is a constant to be chosen sufficiently large.
	By  Lemmas \ref{l22}, \ref{Gg},  and \eqref{cp1},
	\begin{align*}
		\left|
		\int_\Omega
		\rho u\cdot\mathcal B[P-\bar P]\,dx
		\right|
		 \leq
		C\|\sqrt\rho u\|_{L^2}
		\|P-\bar P\|_{L^2}
		 \leq
		CE(t).
	\end{align*}
	Thus, choosing $C_1$ sufficiently large, we have
	\begin{align}\label{modi1}
		\mathcal C_1(t)
		\sim
		E(t)
		\sim
		\|\sqrt\rho u(t)\|_{L^2}^2
		+
		\|\left( P -\bar P\right) (t) \|_{L^2}^2.
	\end{align}
	Multiplying \eqref{w8} by $C_1$ and adding \eqref{p-time}, we obtain
	\begin{align*}
		\frac{d}{dt}\mathcal C_1(t)
		 +
		C_1(2\mu+\lambda)\|\operatorname{div}u\|_{L^2}^2
		+
		C_1\mu\|\operatorname{curl}u\|_{L^2}^2
		 +
		\frac12\|P-\bar P\|_{L^2}^2
		\leq
		C\|\nabla u\|_{L^2}^2.
	\end{align*}
	Using the div--curl estimate \eqref{f6} and choosing $C_1$ sufficiently large, the
	right-hand side can be absorbed by the viscous dissipation term. Hence
	\begin{align*}
		\frac{d}{dt}\mathcal C_1(t)
		+
		C\|\nabla u\|_{L^2}^2
		+
		C\|P-\bar P\|_{L^2}^2
		\leq0,
	\end{align*}
which together with Poincar\'e's inequality and \eqref{cp1} implies
	\begin{align}\label{modified}
		\frac{d}{dt}\mathcal C_1(t)
		+
		\eta_1\mathcal C_1(t)
		+
		C\|\nabla u\|_{L^2}^2
		+
		C\|P-\bar P\|_{L^2}^2
		\leq0,
	\end{align}
for some sufficiently small $\eta_1>0$.
Multiplying \eqref{modified} by $e^{\eta_1t}$ and integrating the resultant over $[0,T]$ gives
	\begin{align*}
		\sup_{0\leq t\leq T}
		\left( e^{\eta_1t}\mathcal C_1(t)\right)
		+
		C\int_0^T
		e^{\eta_1t}
		\left(
		\|\nabla u \|_{L^2}^2
		+
		\| P-\bar P \|_{L^2}^2
		\right)\,dt
		\leq
		\mathcal C_1(0).
	\end{align*}
This along with \eqref{modi1} and
	$
	\mathcal C_1(0)\leq CD(0)$ concludes \eqref{w1}.
\end{proof}

 The basic estimate obtained in Lemma \ref{lem-basic} controls the kinetic energy and the pressure deviation in $L^2$. To close the bootstrap argument
 \eqref{cp1}$_2$, it remains to derive a uniform estimate for $D(t)$.
 \begin{lemma}\label{lem2}
 Assume that \eqref{cp1} holds. Then
 	\begin{gather}\label{xxxx1}
 		\sup_{0\leq t\leq T}D(t)
 		+
 		\int_0^T
 		\|\sqrt{\rho} \dot u \|_{L^2}^2\,dt
 		\leq
 		C D(0),
\\ \label{I02}
 	A_1(T)
 	=
 	\sup_{0\leq t\leq T}
 	\left( \sigma \|\nabla u(t)\|_{L^2}^2\right)
 	+
 	\int_0^T
 	\sigma \|\sqrt\rho \dot u \|_{L^2}^2\,dt
 		\leq
 		C D(0)
 		+
 		C\int_0^T
 		\sigma \|\nabla u \|_{L^3}^3\,dt,
 	\end{gather}
 	where $
 	\sigma=\sigma(t)\triangleq\min\{1,t\}$.
 	Moreover, there exists a constant $\eta_2>0$ such that
 	\begin{align}\label{fir}
 		&\sup_{0\leq t\leq T}
 		\left( e^{\eta_2t}
 		D(t)\right)
 		+
 		\int_0^T
 		e^{\eta_2t}
 		\|\sqrt{\rho} \dot u \|_{L^2}^2\,dt
 		\leq
 		C D(0).
 	\end{align}
 \end{lemma}

\begin{proof}
We rewrite the momentum equation \eqref{a1}$_2$ as
	\begin{align}\label{w18}
		\rho\dot u
		+
		\nabla(P-\bar P)
		=
		(2\mu+\lambda)\nabla\divv u
		-
		\mu\curl\operatorname{curl}u.
	\end{align}
Multiplying \eqref{w18} by
	$\sigma^m\dot u\ (m\geq0)$ and integrating the resultant over $\Omega$, we obtain
	\begin{align}\label{I0}
		\int_\Omega
		\sigma^m\rho|\dot u|^2\,dx
		= &
		-\int_\Omega
		\sigma^m\dot u\cdot\nabla(P-\bar P)\,dx
	 +
		(2\mu+\lambda)
		\int_\Omega
		\sigma^m\nabla\divv u\cdot\dot u\,dx
	 -
		\mu
		\int_\Omega
		\sigma^m\curl\operatorname{curl}u\cdot\dot u\,dx
		\nonumber\\
		\triangleq &
		I_1+I_2+I_3.
	\end{align}

A direct calculation gives
\begin{align}\label{I1-1}
I_1=-\int_\Omega\sigma^m u_t\cdot\nabla(P-\bar P)\,dx
-\int_\Omega\sigma^m(u\cdot\nabla u)\cdot\nabla(P-\bar P)\,dx\triangleq
I_{11}+I_{12}.
\end{align}
It follows from integration by parts that
	\begin{align}\label{I11-1}
		I_{11}
		=&
		\int_\Omega
		\sigma^m(P-\bar P)\divv u_t\,dx
		\nonumber\\
		=&
		\frac{d}{dt}
		\int_\Omega
		\sigma^m(P-\bar P)\divv u\,dx
		-
		m\sigma^{m-1}\sigma'
		\int_\Omega
		(P-\bar P)\divv u\,dx
		-
		\int_\Omega
		\sigma^m(P-\bar P)_t\divv u\,dx
		\nonumber\\
		=&
		\frac{d}{dt}
		\int_\Omega
		\sigma^m(P-\bar P)\divv u\,dx
		-
		m\sigma^{m-1}\sigma^{\prime}
		\int_\Omega
		(P-\bar P)\divv u\,dx
		\nonumber\\
		&+
		\int_\Omega
		\sigma^m
		u\cdot\nabla(P-\bar P)\divv u\,dx
		+
		\gamma
		\int_\Omega
		\sigma^mP(\divv u)^2\,dx,
	\end{align}
due to
\begin{gather*}
		(P-\bar P)_t
		+
		u\cdot\nabla(P-\bar P)
		+
		\gamma P\divv u
		-
		(\gamma-1)\overline{P\divv u}
		=
		0,\\
		-(\gamma-1)
		\overline{P\divv u}
		\int_\Omega\sigma^m\divv u\,dx
		=
		-(\gamma-1)
		\sigma^m
		\overline{P\divv u}
		\int_{\partial\Omega}u\cdot n\,dS=0.
	\end{gather*}	
Moreover, one derives from integration by parts and \eqref{bian} that
	\begin{align}\label{I12}
		I_{12}
		=&
		\int_\Omega
		\sigma^m(P-\bar P)
		\divv(u\cdot\nabla u)\,dx
		-
		\int_{\partial\Omega}
		\sigma^m(P-\bar P)
		(u\cdot\nabla u)\cdot n\,dS
		\nonumber\\
		=&
		\int_\Omega
		\sigma^m(P-\bar P)
		\nabla u:(\nabla u)^{T}\,dx
		+
		\int_\Omega
		\sigma^m(P-\bar P)
		u\cdot\nabla\divv u\,dx
		+
		\int_{\partial\Omega}
		\sigma^m(P-\bar P)
		u\cdot\nabla n\cdot u\,dS.
	\end{align}
	Notice that
	\begin{align*}
		\int\sigma^m
		u\cdot\nabla(P-\bar P)\divv u\,dx
		+
		\int\sigma^m
		(P-\bar P)u\cdot\nabla\divv u\,dx
		&=
		\int\sigma^m
		u\cdot\nabla\big((P-\bar P)\divv u\big)\,dx
		\nonumber\\
		&=
		-\int\sigma^m
		(P-\bar P)(\divv u)^2dx,
	\end{align*}	
	then, from the trace theorem, the boundary integral in \eqref{I12} can be estimated as
	\begin{align*}
		\left|
		\int_{\partial\Omega}
		\sigma^m(P-\bar P)
		u\cdot\nabla n\cdot u\,dS
		\right|
		\leq
		C\sigma^m
		\|u\|_{L^2(\partial\Omega)}^2
		\leq
		C\sigma^m
		\|\nabla u\|_{L^2}^2,
	\end{align*}
which combined with \eqref{I1-1}--\eqref{I12} implies
	\begin{align}\label{I11}
		I_1
		\leq&
		\frac{d}{dt}
		\int_\Omega
		\sigma^m(P-\bar P)\divv u\,dx
		+
		Cm\sigma^{m-1}\sigma^{\prime}
		\left(
		\|P-\bar P\|_{L^2}^2
		+
		\|\nabla u\|_{L^2}^2
		\right)
		+
		C\sigma^m
		\|\nabla u\|_{L^2}^2.
	\end{align}
	
We next estimate $I_2$. Integration by parts shows that
\begin{align}\label{I20}
I_2=&(2\mu+\lambda)\int_{\partial\Omega}
\sigma^m\divv u\,\dot u\cdot n\,dS-(2\mu+\lambda)
\int_\Omega\sigma^m\divv u\,\divv\dot u\,dx
\nonumber\\
=&(2\mu+\lambda)\int_{\partial\Omega}\sigma^m\divv u\,(u\cdot\nabla u)\cdot n\,dS-\frac{2\mu+\lambda}{2}\frac{d}{dt}\int_\Omega
		\sigma^m(\divv u)^2\,dx+
		\frac{2\mu+\lambda}{2}
		\int_\Omega
		\sigma^m(\divv u)^3\,dx
		\nonumber\\[2mm]
		&-
		(2\mu+\lambda)
		\int_\Omega
		\sigma^m\divv u\,
		\nabla u:(\nabla u)^T\,dx
		+
		\frac{m(2\mu+\lambda)}{2}
		\sigma^{m-1}\sigma'
		\int_\Omega
		(\divv u)^2\,dx.
	\end{align}
	For the boundary term, we obtain from \eqref{bian}, \eqref{ff8}, and
$\divv u=\frac{F+P-\bar P}{2\mu+\lambda}$ that
	\begin{align}\label{3.42}
		 (2\mu+\lambda)
		\int_{\partial\Omega}
		\sigma^m\divv u\,
		(u\cdot\nabla u)\cdot n\,dS
		& =
		-\int_{\partial\Omega}
		\sigma^mF
		u\cdot\nabla n\cdot u\,dS
		-
		\int_{\partial\Omega}
		\sigma^m(P-\bar P)
		u\cdot\nabla n\cdot u\,dS
		\nonumber\\[1mm]
		& \leq
		C\sigma^m
		\|\nabla F\|_{L^2}
		\|\nabla u\|_{L^2}^2
		+
		C\sigma^m
		\|\nabla u\|_{L^2}^2
		\nonumber\\[1mm]
		& \leq
		\frac12
		\sigma^m
		\|\sqrt\rho\dot u\|_{L^2}^2
		+
		C\sigma^m
		\|\nabla u\|_{L^2}^4
		+
		C\sigma^m
		\|\nabla u\|_{L^2}^2.
	\end{align}
This together with \eqref{I20} yields
	\begin{align}\label{I2}
		I_2
		\leq &
		-\frac{2\mu+\lambda}{2}
		\frac{d}{dt}\int_\Omega
		\sigma^m(\divv u)^2\,dx
		+
		\frac12
		\sigma^m
		\|\sqrt\rho\dot u\|_{L^2}^2
	+
		C\sigma^m
		\|\nabla u\|_{L^3}^3
			\nonumber\\[2mm]
		&+
		C\sigma^m
		\|\nabla u\|_{L^2}^4
		+
		C\sigma^m
		\|\nabla u\|_{L^2}^2
		 +
		Cm\sigma^{m-1}\sigma'
		\|\nabla u\|_{L^2}^2.
	\end{align}
	
	For $I_3$, using the boundary condition
	$(\operatorname{curl}u\times n)\big|_{\partial\Omega}=0$ and
	\begin{align*}
		\curl(u\cdot\nabla u)
		=
		u^i\partial_i\operatorname{curl}u
		+
		\nabla u^i\times\partial_i u,
	\end{align*}
	we have
	\begin{align}\label{I3}
		I_3
		= &
		-\mu
		\int_\Omega
		\sigma^m\operatorname{curl}u\cdot\curl\dot u\,dx
		\nonumber\\
		= &
		-\frac{\mu}{2}\frac{d}{dt}
		\int_\Omega
		\sigma^m|\operatorname{curl}u|^2\,dx
		+
		\frac{\mu m}{2}
		\sigma^{m-1}\sigma'
		\int_\Omega|\operatorname{curl}u|^2\,dx
		 \notag\\
		 &\quad-
		\mu
		\int_\Omega
		\sigma^m
		(\nabla u^i\times\partial_i u)\cdot\operatorname{curl}u\,dx
		+
		\frac{\mu}{2}
		\int_\Omega
		\sigma^m|\operatorname{curl}u|^2\divv u\,dx
		\nonumber\\
		\leq &
		-\frac{\mu}{2}\frac{d}{dt}
		\int_\Omega
		\sigma^m|\operatorname{curl}u|^2\,dx
		+
		Cm\sigma^{m-1}\sigma'
		\|\nabla u\|_{L^2}^2
		+
		C\sigma^m
		\|\nabla u\|_{L^3}^3.
	\end{align}	
Consequently, substituting \eqref{I11}, \eqref{I2}, and \eqref{I3} into
\eqref{I0}, we arrive at
	\begin{align}\label{I01}
		&\frac{d}{dt}\left[
		(2\mu+\lambda)
		\int_\Omega
		\sigma^m(\divv u)^2\,dx
		+
		\mu
		\int_\Omega
		\sigma^m|\operatorname{curl}u|^2\,dx
		-
		2
		\int_\Omega
		\sigma^m(P-\bar P)\divv u\,dx
		\right]
		+
		\int_\Omega
		\sigma^m\rho|\dot u|^2\,dx
		\nonumber\\[2mm]
		&\leq
		Cm\sigma^{m-1}\sigma'
		\left(
		\|P-\bar P\|_{L^2}^2
		+
		\|\nabla u\|_{L^2}^2
		\right)
		+
		C\sigma^m
		\left(
		\|\nabla u\|_{L^2}^2
		+
		\|\nabla u\|_{L^3}^3
		+
		\|\nabla u\|_{L^2}^4
		\right).
	\end{align}
	
Taking $m=0$, one deduces from \eqref{I01} that
	\begin{align}\label{w30}
\frac{d}{dt}\Psi(t)+
		\|\sqrt\rho\dot u\|_{L^2}^2
		\leq
		C
		\left(
		\|\nabla u\|_{L^3}^3
		+
		\|\nabla u\|_{L^2}^4
		+
		\|\nabla u\|_{L^2}^2
		\right),
	\end{align}
	where
	\begin{align*}
		\Psi(t)
		\triangleq
		\int_\Omega
		\left[
		(2\mu+\lambda)(\divv u)^2
		+
		\mu|\operatorname{curl}u|^2
		-
		2(P-\bar P)\divv u
		\right](t)\, dx.
	\end{align*}
To control the cubic term in \eqref{w30}, it follows from \eqref{f12} that
	\begin{align}\label{w31}
		\|\nabla u\|_{L^3}^3
		\leq &
		C
		\|\sqrt\rho\dot u\|_{L^2}^{3/2}
		\left(
		\|\nabla u\|_{L^2}
		+
		\|P-\bar P\|_{L^2}
		\right)^{3/2}
		 +
		C
		\left(
		\|\nabla u\|_{L^2}^3
		+
		\|P-\bar P\|_{L^3}^3
		\right)
		\nonumber\\[2mm]
		\leq &
		\frac12
		\|\sqrt\rho\dot u\|_{L^2}^2
		+
		C
		\left(
		\|\nabla u\|_{L^2}^2
		+
		\|\nabla u\|_{L^2}^6
		+
		\|P-\bar P\|_{L^2}^2
		\right).
	\end{align} 	
Thus, inserting \eqref{w31} into \eqref{w30} leads to
	\begin{align}\label{lwz3}
		\frac{d}{dt}\Psi(t)
		+
		\frac12
		\|\sqrt\rho\dot u\|_{L^2}^2
		\leq
		C
		\left(
		\|\nabla u\|_{L^2}^2
		+
		\|\nabla u\|_{L^2}^6
		+
		\|P-\bar P\|_{L^2}^2
		\right).
	\end{align}
Define
\begin{align*}
	\mathcal C_2(t)
	\triangleq
	\Psi(t)
	+
	C_2\mathcal C_1(t),
\end{align*}
where $C_2>0$ is to be chosen.
	By \eqref{modi1},
	Poincar\'e's inequality, and Young's inequality,
	we can choose $C_2$ sufficiently large so that
	\begin{align}\label{Phi-equivalence}
		\mathcal C_2(t)
		\sim
		D(t).
	\end{align}
Multiplying \eqref{modified} by a sufficiently large $C_2$ and adding the resultant to \eqref{lwz3}, we derive
\begin{align}\label{lwz4}
\frac{d}{dt}\mathcal C_2(t)+C\left(\|\nabla u\|_{L^2}^2
+\|P-\bar P\|_{L^2}^2+\|\sqrt\rho\dot u\|_{L^2}^2\right)
\leq
C\|\nabla u\|_{L^2}^6.
\end{align}
Integrating \eqref{lwz4} over $[0,T]$ together with
\eqref{basic1} and \eqref{Phi-equivalence} yields
	\begin{align}\label{lwz5}
	\sup_{0\leq t\leq T}
D(t)+\int_0^T\|\sqrt\rho\dot u \|_{L^2}^2\,dt
&\leq
C D(0)+C\int_0^T\|\nabla u \|_{L^2}^6\,dt
		\nonumber\\
		&\leq
		C D(0)
		+C\sup_{0\leq t\leq T}
		\|\nabla u(t)\|_{L^2}^2
		\left[
		\sup_{0\leq t\leq T}
	\|\nabla u(t)\|_{L^2}^2
		\int_0^T
		\|\nabla u \|_{L^2}^2\,dt
		\right].
	\end{align}
	By the bootstrap assumption \eqref{cp1}$_{2}$, we may choose
	$\varepsilon>0$ sufficiently small such that
	\begin{align*}
		C
		\sup_{0\leq t\leq T}
		\|\nabla u(t)\|_{L^2}^2
		\int_0^T
		\|\nabla u \|_{L^2}^2\,dt
		\leq
		2C\varepsilon
		\leq
		\frac12,
	\end{align*}
which combined with \eqref{lwz5} gives \eqref{xxxx1}.
	
We next take $m=1$ in \eqref{I01}. Using Young's inequality, \eqref{cp1}, \eqref{basic1}, and \eqref{xxxx1}, we obtain
\begin{align*}
	A_1(T)&\leq C\sup_{0\leq t\leq T}\left( \|\nabla u(t)\|_{L^2}^2+\|\left( P-\bar P\right) (t)\|_{L^2}^2\right) +
		C
		\int_0^{T}
		\left(
		\|  P-\bar P \|_{L^2}^2
		+
		\|\nabla u \|_{L^2}^2
		\right)dt
		\nonumber\\
		&\quad+
		C
		\int_0^T
		\sigma
		\left(
		\|\nabla u \|_{L^2}^2
		+
		\|\nabla u \|_{L^3}^3
		+
		\|\nabla u \|_{L^2}^4
		\right)dt\notag\\
		&\leq
		C D(0)
		+
		C
		\int_0^T
		\sigma
		\|\nabla u \|_{L^3}^3\,dt,
	\end{align*}
proving \eqref{I02}.
	
It remains to show \eqref{fir}. By the bootstrap
assumption \eqref{cp1}$_2$ and choosing $\varepsilon>0$ sufficiently small such that
\begin{align*}
	C\|\nabla u\|_{L^2}^6
	\leq
	C(2\varepsilon)^2\|\nabla u\|_{L^2}^2
	\leq
	\frac C2  \|\nabla u\|_{L^2}^2,
\end{align*}
which together with \eqref{lwz4} indicates that
\begin{align}\label{cm2-decay-local}
	\frac{d}{dt}\mathcal C_2(t)
	+
CD(t)
	+
	C\|\sqrt\rho \dot u\|_{L^2}^2
	\leq0.
\end{align}
By \eqref{Phi-equivalence}, we may choose $\eta_2>0$ sufficiently small
such that
\begin{align}\label{cm2-decay}
	\frac{d}{dt}\mathcal C_2(t)
	+
	\eta_2\mathcal C_2(t)
	+
C\|\sqrt\rho \dot u\|_{L^2}^2
	\leq0.
\end{align}
Multiplying \eqref{cm2-decay} by $e^{\eta_2t}$ and integrating the resultant over
$[0,T]$, we have
\begin{align*}
\sup_{0\leq t\leq T}
\left( e^{\eta_2t}\mathcal C_2(t)\right)
+
C\int_0^T
e^{\eta_2t}
\|\sqrt\rho \dot u \|_{L^2}^2\,dt
\leq
\mathcal C_2(0)
\leq
CD(0).
\end{align*}
This along with \eqref{Phi-equivalence} yields the desired \eqref{fir}.
\end{proof}

Next, we show that the term $\int_0^T\sigma\|\nabla u\|_{L^3}^3\,dt$
in \eqref{I02} can be controlled by the bootstrap quantity $A_1(T)$ with the aid of the smallness assumption in \eqref{cp1}.
\begin{lemma}\label{Aclosed}
Assume that \eqref{cp1} holds. Then
	\begin{align}\label{A1-final-closed}
		A_1(T)
		\leq
		C D(0).
	\end{align}
\end{lemma}

\begin{proof}
	By \eqref{basic1}, \eqref{I02},  and \eqref{w31},
	we have
	\begin{align}\label{A1}
		A_1(T)
		&\leq
		C D(0)
		+
		C\int_0^T
		\sigma
		\|\sqrt\rho \dot u \|_{L^2}^{3/2}
		\|\nabla u \|_{L^2}^{3/2}\,dt
		+
		C\int_0^T
		\sigma
		\|\sqrt\rho \dot u \|_{L^2}^{3/2}
		\|  P-\bar P \|_{L^2}^{3/2}\,dt	\nonumber\\
		&\quad
		+
		C\int_0^T
		\sigma
		\|\nabla u \|_{L^2}^3\,dt
		+
		C\int_0^T
		\sigma
		\|  P-\bar P \|_{L^3}^3\,dt\notag\\
		&\leq
		C D(0)
		+
		C\int_0^T
		\sigma
		\|\sqrt\rho \dot u \|_{L^2}^{3/2}
		\|\nabla u \|_{L^2}^{3/2}\,dt+
		C\int_0^T
		\sigma
		\|\sqrt\rho \dot u \|_{L^2}^{3/2}
		\|  P-\bar P \|_{L^2}^{3/2}\,dt
		+
		C\int_0^T
		\sigma
		\|\nabla u \|_{L^2}^3\,dt\notag\\[2mm]
		&\triangleq
		CD(0)+H_1+H_2+H_3.
	\end{align}

Observe that
		\begin{align*}
	\sigma
	\|\sqrt\rho \dot u\|_{L^2}^{3/2}
	\|\nabla u\|_{L^2}^{3/2}
	=
	\left(
	\sigma\|\sqrt\rho \dot u\|_{L^2}^2
	\right)^{3/4}
	\left(
	\sigma\|\nabla u\|_{L^2}^2
	\right)^{1/4}
	\|\nabla u\|_{L^2},
\end{align*}
hence it follows from the definition of $A_1(T)$, H\"older's inequality, and \eqref{cp1} that
	\begin{align}\label{A11}
		H_1
		&\leq
		C
		\left(
		\sup_{0\leq t\leq T}
		\sigma(t)\|\nabla u(t)\|_{L^2}^2
		\right)^{1/4}
		\left(
		\int_0^T
		\sigma
		\|\sqrt\rho \dot u\|_{L^2}^2\,dt
		\right)^{3/4}
		\left(
		\int_0^T
		\|\nabla u\|_{L^2}^4\,dt
		\right)^{1/4}\notag\\
		&\leq
		C
		\left(
		\sup_{0\leq t\leq T}
		\sigma(t)\|\nabla u(t)\|_{L^2}^2
		\right)^{1/4}
		\left(
		\int_0^T
		\sigma
		\|\sqrt\rho \dot u\|_{L^2}^2\,dt
		\right)^{3/4}\left(
		\sup_{0\leq t\leq T}
		\|\nabla u(t)\|_{L^2}^2
		\int_0^T
		\|\nabla u \|_{L^2}^2\,dt\right)^{1/4}\notag\\[2mm]
		&\leq
		C A_1(T)\varepsilon^{1/4}.
	\end{align}	
Moreover, from \eqref{basic1}, H\"older's inequality, and Young's inequality, one has that, for any $\delta>0$,
	\begin{align}\label{A12}
		H_2
		&\leq
		C
		\left(
		\int_0^T
		\sigma
		\|\sqrt\rho \dot u\|_{L^2}^2\,dt
		\right)^{3/4}
		\left(
		\int_0^T
		\sigma
		\|P-\bar P\|_{L^2}^6\,dt
		\right)^{1/4}
		\nonumber\\
		&\leq
		C
		A_1(T)^{3/4}
		\left(
		\sup_{0\leq t\leq T}
		\|(P-\bar P)(t)\|_{L^2}^2
		\right)^{1/2}
		\left(
		\int_0^T
		\|P-\bar P\|_{L^2}^2\,dt
		\right)^{1/4}\notag\\[2mm]
		&\leq
	    C A_1(T)^{3/4}D(0)^{3/4}\leq
		\delta A_1(T)
		+
		C  D(0)^3.
	\end{align}
Similarly, one gets that
	\begin{align}\label{A13}
		H_3
		&=
		C\int_0^T
		\left(
		\sigma\|\nabla u\|_{L^2}^2
		\right)^{1/2}
		\sigma^{1/2}
		\|\nabla u\|_{L^2}^2\,dt
		\nonumber\\
		&\leq
		C
		\left(
		\sup_{0\leq t\leq T}
		\sigma(t)\|\nabla u(t)\|_{L^2}^2
		\right)^{1/2}
		\int_0^T
		\|\nabla u\|_{L^2}^2\,dt\notag\\[2mm]
	&\leq
	C D(0)A_1(T)^{1/2}\leq
		\delta A_1(T)
		+
		C  D(0)^2.
	\end{align}
Combining \eqref{A1}--\eqref{A13}, choosing $\delta>0$ sufficiently small,
and using $D(0)<1$, we obtain
	\begin{align*}
		A_1(T)
		&\leq
		CD(0)
		+
		C\varepsilon^{1/4}A_1(T)
		+
		\frac12A_1(T)
		+
		C\bigl(D(0)^2+D(0)^3\bigr)
		\notag\\
		&\leq
		CD(0)
		+
		\left(
		C\varepsilon^{1/4}
		+
		\frac12
		\right)A_1(T),
	\end{align*}
which yields \eqref{A1-final-closed} as long as $\varepsilon$ is chosen sufficiently small.
\end{proof}

To establish the uniform upper bound of the density via Zlotnik's
inequality, it is essential to obtain an estimate for $\nabla\dot u$.
The main difficulty comes from the boundary behavior of the material derivative. In fact, although $u$ satisfies the slip boundary condition, $\dot u$ does not inherit the corresponding boundary structure. Therefore, a boundary correction term is
 introduced to compensate for the non-vanishing boundary contribution. The following lemma provides the required estimate.
\begin{lemma}\label{lemb}
Assume that \eqref{cp1} holds. Then
	\begin{align}\label{mB1}
		B_1(T)
		=
		\sup_{0\leq t\leq T}
		\left( \sigma \|\sqrt\rho\dot u(t)\|_{L^2}^2\right)
		+
		\int_0^T
		\sigma \|\nabla\dot u \|_{L^2}^2\,dt
		\leq
		C D(0).
	\end{align}
	Moreover, if $T\geq1$, then there exists a constant $\eta_3>0$ such that
	\begin{align}\label{w41}
		\sup_{1\leq t\leq T}
		\left( e^{\eta_3t}
		\|\sqrt\rho\dot u(t)\|_{L^2}^2\right)
		+
		\int_1^T
		e^{\eta_3t}
		\|\nabla\dot u \|_{L^2}^2\,dt
		\leq
		C D(0).
	\end{align}
\end{lemma}
 \begin{proof}
 	By the definition of the effective viscous flux, we rewrite the momentum equation \eqref{a1}$_2$ as
 	\begin{align}\label{3.41}
 	\rho\dot u
 	=
 	\nabla F
 	-
 	\mu\curl \operatorname{curl}u.
 \end{align}
 Let $m\geq0$. Applying the operator
$ \partial_t+\operatorname{div}(u\,\cdot)
$ to \eqref{3.41}, multiplying the resultant $j$-th equation by
 $\sigma^m\dot u^j$, and summing over $j$, we
 obtain that
 \begin{align}\label{J0}
 	&\frac{d}{dt}\left(
 	\frac{\sigma^m}{2}
 	\int_\Omega\rho|\dot u|^2dx
 	\right)
 	-
 	\frac{m}{2}
 	\sigma^{m-1}\sigma'
 	\int_\Omega\rho|\dot u|^2dx
 	\nonumber\\
 	&=
 	\int_\Omega\sigma^m
 	\left[
 	\dot u\cdot\nabla F_t
 	+
 	\dot u^j
 	\partial_i
 	\left(
 	u^i\partial_jF
 	\right)
 	\right]dx
 	-
 	\mu
 	\int_\Omega\sigma^m
 	\left[
 	\dot u\cdot\operatorname{curl}\operatorname{curl}u_t
 	+
 	\dot u^j
 	\partial_i
 	\left(
 	u^i(\operatorname{curl}\operatorname{curl}u)^j
 	\right)
 	\right]dx
 	\nonumber\\
 	&\triangleq
 	J_1+J_2.
 \end{align}

Now we estimate $J_1$. Integrating by parts and using
$(u\cdot n)\big|_{\partial\Omega}=0$, we have
\begin{align}\label{J1-1}
	J_1
	=
	\int_{\partial\Omega}
	\sigma^mF_t\dot u\cdot n\,dS
	-
	\int_\Omega
	\sigma^mF_t\divv \dot u\,dx
	-
	\int_\Omega
	\sigma^m
	u^i\partial_i\dot u^j\partial_jF\,dx\triangleq J_1^{1}+J_1^{2}+J_1^{3},
\end{align}
due to
\begin{align*}
	\int_\Omega
	\sigma^m
	\dot u^j
	\partial_i
	\left(
	u^i\partial_jF
	\right)dx
	=
	-\int_\Omega
	\sigma^m
	u^i\partial_i\dot u^j
	\partial_jF\,dx
	+
	\int_{\partial\Omega}
	\sigma^m
	\dot u^j u^i\partial_jF n_i\,dS
	=
	-\int_\Omega
	\sigma^m
	u^i\partial_i\dot u^j
	\partial_jF\,dx.
\end{align*}
One needs to control each term $J_1^{i}\ (i=1,2,3)$. It follows from \eqref{bian} that
\begin{align*}
	\dot u\cdot n
	=
	u_t\cdot n
	+
	(u\cdot\nabla u)\cdot n
	=
	-u\cdot\nabla n\cdot u
	\qquad
	\text{on }\partial\Omega.
\end{align*}
Hence
\begin{align}
J_1^{1}
	=
	-\frac{d}{dt}
	\left(
	\sigma^m\mathcal J(t)
	\right)
	+
	m\sigma^{m-1}\sigma'
	\mathcal J(t)
	+
	\sigma^m
	\int_{\partial\Omega}
	F
	\frac{d}{dt}
	(u\cdot\nabla n\cdot u)\,dS,
	\label{boundary-time-derivative}
\end{align}
where $
	\mathcal J(t)
	\triangleq
	\int_{\partial\Omega}
	(u\cdot\nabla n\cdot u)F\,dS$.
To estimate the third term on the right-hand side of \eqref{boundary-time-derivative}, note that
\begin{align*}
	u_{t} \cdot \nabla n \cdot u+u \cdot \nabla n \cdot u_{t} =u_{t} \cdot\big[\nabla n+(\nabla n)^{\top}\big] u,
\end{align*}
which gives
\begin{align*}
	\frac{d}{dt}
	\left(
	u\cdot\nabla n\cdot u
	\right)
	 =h\cdot u_t,\quad \text{where}\quad h
	 =
	 \big(
	 \nabla n+(\nabla n)^{\mathsf T}
	 \big)u.
\end{align*}
Then we have
\begin{align}
	J_1^{1}&=
	-\frac{d}{dt}
	\left(
	\sigma^m\mathcal J(t)
	\right)
	+
	m\sigma^{m-1}\sigma'
	\mathcal J(t)+
	\int_{\partial\Omega}
	\sigma^mFh\cdot u_t\,dS\notag\\
	&=-\frac{d}{dt}
	\left(
	\sigma^m\mathcal J(t)
	\right)
	+
	m\sigma^{m-1}\sigma'
	\mathcal J(t)+\int_{\partial\Omega}
	\sigma^mFh\cdot\dot u\,dS
	-
	\int_{\partial\Omega}
	\sigma^mFh\cdot(u\cdot\nabla u)\,dS.
	\label{bo11}
\end{align}
For the estimates of the boundary terms, we fix a smooth extension of the
unit outward normal vector $n$ from $\partial\Omega$ to a neighborhood of
$\partial\Omega$ (and, if necessary, to $\overline\Omega$). Accordingly, we extend
$u^\perp$ and $h$ to the interior of $\Omega$.
It follows \eqref{hi} that
\begin{align*}
	u\cdot\nabla u
	=
	\big(
	(u^\perp\times n)\cdot\nabla
	\big)u
	\quad
	\text{on }\partial\Omega.
\end{align*}
Consequently,
\begin{align}
	-\int_{\partial\Omega}
	Fh\cdot(u\cdot\nabla u)\,dS
& =
	-\int_{\partial\Omega}
	Fh^i
	\big(
	\nabla u^i\times u^\perp
	\big)\cdot n\,dS\notag\\
 &=
	-\int_\Omega
	\divv
	\big(
	Fh^i
	\nabla u^i\times u^\perp
	\big)dx\notag\\
	&=-\int_\Omega
	\nabla u^i\times u^\perp
	\cdot\nabla(Fh^i)\,dx
	-\int_\Omega
	Fh^i
	\divv
	\big(
	\nabla u^i\times u^\perp
	\big)dx\notag\\
	&=
	-\int_\Omega
	\nabla u^i\times u^\perp
	\cdot\nabla(Fh^i)\,dx
	+
	\int_\Omega
	Fh^i
	\operatorname{curl}u^\perp\cdot\nabla u^i\,dx.
	\label{bo1}
\end{align}
Substituting \eqref{bo1} into
\eqref{bo11}, we arrive at
\begin{align}
	J_1^{1}
&=-\frac{d}{dt}
\left(
\sigma^m\mathcal J(t)
\right)
+
m\sigma^{m-1}\sigma'
\mathcal J(t)+\int_{\partial\Omega}
\sigma^mFh\cdot\dot u\,dS\notag\\
&\quad-\int_\Omega 	\sigma^m
\nabla u^i\times u^\perp
\cdot\nabla(Fh^i)\,dx
+
\int_\Omega 	\sigma^m
Fh^i
\operatorname{curl}u^\perp\cdot\nabla u^i\,dx.
\label{bo12}
\end{align}

For the boundary term in \eqref{bo12}, one deduces from the trace
theorem and Lemmas \ref{l24} and \ref{xu} that
\begin{align}
	 \left|
	\int_{\partial\Omega}
	\sigma^mFh\cdot\dot u\,dS
	\right|
	& \leq
	C\sigma^m
	\|F\|_{H^1}
	\|u\|_{H^1}
	\left(
	\|\nabla\dot u\|_{L^2}
	+
	\|\nabla u\|_{L^2}^2
	\right)
	\nonumber\\
	& \leq
	C\sigma^m
	\|\sqrt\rho\dot u\|_{L^2}
	\|\nabla u\|_{L^2}
	\left(
	\|\nabla\dot u\|_{L^2}
	+
	\|\nabla u\|_{L^2}^2
	\right)
	\nonumber\\[2mm]
	& \leq
	C\delta\sigma^m
	\|\nabla\dot u\|_{L^2}^2
	+
	C \left( \sigma^m
	\|\nabla u\|_{L^2}^2
	\|\sqrt\rho\dot u\|_{L^2}^2
	+
	\|\nabla u\|_{L^2}^4\right).
	\label{boundary1}
\end{align}
For the two volume integrals in
\eqref{bo12}, by
Lemmas \ref{l24}, \ref{xu}, H\"older's inequality, and Young's inequality, we obtain
\begin{align}
	&\left|
	\int_\Omega
	\sigma^m
	\nabla u^i\times u^\perp
	\cdot\nabla(Fh^i)\,dx
	\right|
+
	\left|
	\int_\Omega
	\sigma^m
	Fh^i
	\operatorname{curl}u^\perp\cdot\nabla u^i\,dx
	\right|
	\nonumber\\
	&\leq
	C\sigma^m
	\|\nabla u\|_{L^2}^3
	\|\nabla F\|_{L^6}
	+
	C\sigma^m
	\|\nabla u\|_{L^4}^2
	\|\nabla u\|_{L^2}
	\|F\|_{L^3}\notag\\[1mm]
	&\leq
	C\sigma^m
	\|\nabla u\|_{L^2}^3
	\left(
	\|\nabla\dot u\|_{L^2}
	+
	\|\nabla u\|_{L^2}^2
	\right)+C\sigma^m
	\|\nabla u\|_{L^4}^2
	\|\nabla u\|_{L^2}
	\|\sqrt\rho\dot u\|_{L^2}
	\nonumber\\[1mm]
	&\leq
C\delta\sigma^m
	\|\nabla\dot u\|_{L^2}^2
	+
	C \sigma^m\left(
	\|\nabla u\|_{L^2}^4
	+
	\|\nabla u\|_{L^2}^6
+
	\|\nabla u\|_{L^4}^4
	+
	\|\nabla u\|_{L^2}^2
	\|\sqrt\rho\dot u\|_{L^2}^2\right),
	\label{boundary2}
\end{align}
while the trace inequality, Poincar\'e's inequality for $F$, Young's inequality, and \eqref{xxxx1} yield that
\begin{align}
 \left|
	\mathcal J(t)
	\right|
 \leq
	C
	\|F\|_{H^1}
	\|u\|_{H^1}^2
 \leq
	\frac{1}{8}
	\|\sqrt\rho\dot u\|_{L^2}^2
	+
	C
	\|\nabla u\|_{L^2}^4.
	\label{boundary3}
\end{align}
So, inserting \eqref{boundary1}--\eqref{boundary3} into
\eqref{bo12}, one gets
\begin{align*}
	J_1^{1}
	 &\leq-\frac{d}{dt}
	\left(
	\sigma^m\mathcal J(t)
	\right)
	+ Cm\sigma^{m-1}\sigma^{\prime}
	\left(
	\|\sqrt\rho\dot u\|_{L^2}^2
	+
	\|\nabla u\|_{L^2}^2
	\right)\notag\\[1mm]
	&\quad+C\delta\sigma^m
	\|\nabla\dot u\|_{L^2}^2
	+
	C \sigma^m\left(
	\|\nabla u\|_{L^2}^2
	\|\sqrt\rho\dot u\|_{L^2}^2
 +
	\|\nabla u\|_{L^4}^4
	+
	\|\nabla u\|_{L^2}^2\right).
\end{align*}

Since
\begin{gather*}
	(P-\bar P)_t
	+
	u\cdot\nabla(P-\bar P)
	+
	\gamma P\divv u
	-
	(\gamma-1)\overline{P\divv u}
	=
	0,\\
	\divv\dot u
	=
	\divv u_t+\divv(u\cdot\nabla u)
=
	\divv u_t
	+
	u\cdot\nabla\divv u
	+
	\partial_i u^j\partial_j u^i,
\end{gather*}
we derive an evolution equation for $F$:
\begin{align*}
	F_t+u\cdot\nabla F
	=
	(2\mu+\lambda)\divv\dot u
	-
	(2\mu+\lambda)
	\partial_i u^j\partial_j u^i
+
	\gamma P\divv u
	-
	(\gamma-1)\overline{P\divv u},
\end{align*}
which implies that
\begin{align}\label{J22}
	J_1^{2}
	=&
	-(2\mu+\lambda)
	\int_\Omega
	\sigma^m(\divv\dot u)^2dx
	+
	(2\mu+\lambda)
	\int_\Omega
	\sigma^m
	\partial_i u^j\partial_j u^i
	\divv\dot u\,dx
	-
	\gamma
	\int_\Omega
	\sigma^m
	P\divv u\divv\dot u\,dx
	\nonumber\\
	&+
	(\gamma-1)
	\sigma^m
	\overline{P\divv u}
	\int_\Omega\divv\dot u\,dx
	+
	\int_\Omega
	\sigma^m
	u\cdot\nabla F
	\divv\dot u\,dx\notag\\
	\leq&-(2\mu+\lambda)
	\int_\Omega
	\sigma^m|\divv\dot u|^2dx+C\sigma^m\left( \delta
	\|\nabla\dot u\|_{L^2}^2
	+
	\|\nabla u\|_{L^2}^2+\|\nabla u\|_{L^4}^4\right)+
	\int_\Omega
	\sigma^m
	u\cdot\nabla F
	\divv\dot u\,dx,
\end{align}
where we have used Young's inequality and the following facts
\begin{gather*}
	\left|
	\overline{P\divv u}
	\right|
	\leq
	\frac{1}{|\Omega|}
	\|P\|_{L^2}
	\|\divv u\|_{L^2}
\leq
	C\|\nabla u\|_{L^2},
\\
	(\gamma-1)
	\sigma^m
	\left|
	\overline{P\divv u}
	\int_\Omega\divv\dot u\,dx
	\right|
	 \leq
	C\sigma^m
	\|\nabla u\|_{L^2}
	\|\divv\dot u\|_{L^2}
	 \leq
	C\delta\sigma^m
	\|\nabla\dot u\|_{L^2}^2
	+
	C \sigma^m
	\|\nabla u\|_{L^2}^2.
\end{gather*}

For the last term on the right-hand side of \eqref{J22} and $J_1^{3}$, we have
\begin{align*}
	&\left|
	\int_\Omega
	\sigma^m
	u\cdot\nabla F
	\divv\dot u\,dx
	-
	\int_\Omega
	\sigma^m
	u^i\partial_i\dot u^j
	\partial_jF\,dx
	\right|
	\nonumber\\
	& \leq
	C\sigma^m
	\|u\|_{L^6}
	\|\nabla F\|_{L^3}
	\|\nabla\dot u\|_{L^2}
	\nonumber\\
	& \leq
	\frac{1}{16}\sigma^m
	\|\nabla\dot u\|_{L^2}^2
	+
	C \sigma^m
	\|\nabla u\|_{L^2}^2
	\|\nabla F\|_{L^3}^2
	\nonumber\\
	& \leq
	\frac{1}{16}\sigma^m
	\|\nabla\dot u\|_{L^2}^2
	+
	C \sigma^m
	\|\nabla u\|_{L^2}^2
	\|\nabla F\|_{L^2}
	\|\nabla F\|_{L^6}\nonumber\\
	& \leq
	\frac{1}{16}\sigma^m
	\|\nabla\dot u\|_{L^2}^2
+C\sigma^m
	\|\nabla u\|_{L^2}^2
	\|\sqrt\rho\dot u\|_{L^2}
	\left(
	\|\nabla\dot u\|_{L^2}
	+
	\|\nabla u\|_{L^2}^2
	\right)\notag\\
	& \leq
	\frac18\sigma^m
	\|\nabla\dot u\|_{L^2}^2
	+
	C  \sigma^m
	\left( \|\nabla u\|_{L^2}^4
	\|\sqrt\rho\dot u\|_{L^2}^2
	+
	\|\nabla u\|_{L^2}^2\right).
\end{align*}
Consequently, putting the above estimates on $J_1^{i}$ $(i=1, 2, 3)$  into \eqref{J1-1}, we obtain
\begin{align}\label{J1-final}
	J_1
	\leq&
	-\frac{d}{dt}
	\left(
	\sigma^m\mathcal J(t)
	\right)
	-
	(2\mu+\lambda)
	\sigma^m
	\|\divv \dot u\|_{L^2}^2
	+
	Cm\sigma^{m-1}\sigma^{\prime}
	\left(
	\|\sqrt\rho\dot u\|_{L^2}^2
	+
	\|\nabla u\|_{L^2}^2
	\right)
	+
	C\delta\sigma^m
	\|\nabla\dot u\|_{L^2}^2
	\nonumber\\
	&+
	C \sigma^m
	\big[
	(\|\nabla u\|_{L^2}^2+\|\nabla u\|_{L^2}^4)
	\|\sqrt\rho\dot u\|_{L^2}^2
	+
	\|\nabla u\|_{L^2}^2
	+
	\|\nabla u\|_{L^4}^4
	\big].
\end{align}

Next we estimate $J_2$. Note that
\begin{align*}
\operatorname{div}\left(u(\operatorname{curl} \operatorname{curl}u)^{j}\right)=\operatorname{div} u(\operatorname{curl} \operatorname{curl}u)^{j}+u^{i} \partial_{i}(\operatorname{curl} \operatorname{curl}u)^{j},
\end{align*}
one finds
\begin{align}\label{J2s}
	J_2
	&=
	-\mu\int_\Omega
	\sigma^m
	\dot u\cdot
	\operatorname{curl}\operatorname{curl}u_t\,dx
	-
	\mu\int_\Omega
	\sigma^m
	\dot u\cdot
	\operatorname{curl}\operatorname{curl}u
	\operatorname{div}u\,dx
	 -
	\mu\int_\Omega
	\sigma^m
	u^i\dot u\cdot
	\operatorname{curl}(\partial_i\operatorname{curl}u)\,dx\notag\\
	&\triangleq J_2^{1}+J_2^{2}+J_2^{3}.
\end{align}
Thanks to
\begin{align*}
	\operatorname{curl}u_t
	=
	\operatorname{curl}u_t
=
	\operatorname{curl}\dot u
	-
	\operatorname{curl}(u\cdot\nabla u)=\curl \dot u-u^i\partial_i \operatorname{curl}u-\na u^i\times \pa_iu,
\end{align*}
one sees that
\begin{align*}
	J_2^{1}&=-\mu\int_\Omega
	\sigma^m
	\dot u\cdot
	\operatorname{curl}\operatorname{curl}u_t\,dx\notag\\
	&=-\mu\int_\Omega
	\sigma^m
	\operatorname{curl}\dot u\cdot\operatorname{curl}u_t\,dx\notag\\
	&=
	-\mu\int_\Omega
	\sigma^m
	|\operatorname{curl}\dot u|^2\,dx
	+
	\mu\int_\Omega
	\sigma^m
	\operatorname{curl}\dot u
	\cdot
	\operatorname{curl}(u\cdot\nabla u)\,dx
	\nonumber\\
	& =
	-\mu\int_\Omega
	\sigma^m
	|\operatorname{curl}\dot u|^2\,dx
	+
	\mu\int_\Omega
	\sigma^m
	u^i\partial_i\operatorname{curl}u
	\cdot
	\operatorname{curl}\dot u\,dx
	+
	\mu\int_\Omega
	\sigma^m
	\operatorname{curl}\dot u
	\cdot
	\left(
	\nabla u^i\times\partial_i u
	\right)dx.
\end{align*}
For $J_2^{2}$, integration by parts and the fact
\begin{align*}
	\dot u\cdot\operatorname{curl}\operatorname{curl}u
	=
	\operatorname{div}(\operatorname{curl}u\times\dot u)
	+
	\operatorname{curl}u\cdot\operatorname{curl}\dot u
\end{align*}
imply that
\begin{align*}
	J_2^{2}
	&=
	-\mu\int_\Omega
	\sigma^m
	\operatorname{div}u\,
	\operatorname{div}(\operatorname{curl}u\times\dot u)\,dx
	-
	\mu\int_\Omega
	\sigma^m
	\operatorname{div}u\,
	\operatorname{curl}u\cdot\operatorname{curl}\dot u\,dx\notag\\
	&=
	\mu\int_\Omega
	\sigma^m
	(\operatorname{curl}u\times\dot u)
	\cdot\nabla\operatorname{div}u\,dx
	-
	\mu\int_\Omega
	\sigma^m
	\operatorname{div}u\,
	\operatorname{curl}u\cdot\operatorname{curl}\dot u\,dx,
\end{align*}
 where the boundary terms vanish due to
$
(\operatorname{curl}u\times\dot u)\cdot n
=
(n\times\operatorname{curl}u)\cdot\dot u
=
0$
 on $\partial\Omega$.
For $J_2^{3}$, since
\begin{gather*}
	\dot u\cdot
	\operatorname{curl}(\partial_i\operatorname{curl}u)
	=
	\operatorname{div}
	\left(
	\partial_i\operatorname{curl}u\times\dot u
	\right)
	+
	\partial_i\operatorname{curl}u\cdot
	\operatorname{curl}\dot u,\\
	\partial_i\operatorname{curl}u\times\dot u
	=
	\partial_i(\operatorname{curl}u\times\dot u)
	-
	\operatorname{curl}u\times\partial_i\dot u,
\end{gather*}
one deduces that
\begin{align*}
J_2^{3}
	&=
	-\mu\int_\Omega
	\sigma^m
	u^i
	\operatorname{div}
	\left(
	\partial_i\operatorname{curl}u\times\dot u
	\right)dx
	-
	\mu\int_\Omega
	\sigma^m
	u^i\partial_i\operatorname{curl}u
	\cdot
	\operatorname{curl}\dot u\,dx\notag\\
	&=	-\mu\int_\Omega
	\sigma^m
	u\cdot\nabla
	\operatorname{div}
	(\operatorname{curl}u\times\dot u)\,dx
	+
	\mu\int_\Omega
	\sigma^m
	u^i
	\operatorname{div}
	\left(
	\operatorname{curl}u\times\partial_i\dot u
	\right)dx-
	\mu\int_\Omega
	\sigma^m
	u^i\partial_i\operatorname{curl}u
	\cdot
	\operatorname{curl}\dot u\,dx\notag\\
	&=-\mu\int_\Omega
	\sigma^m
	u\cdot\nabla
	\operatorname{div}
	(\operatorname{curl}u\times\dot u)\,dx+\mu\int_\Omega
	\sigma^m
	\left(
	\operatorname{curl}u\times\nabla u^i
	\right)
	\cdot\partial_i\dot u\,dx-
	\mu\int_\Omega
	\sigma^m
	u^i\partial_i\operatorname{curl}u
	\cdot
	\operatorname{curl}\dot u\,dx.
\end{align*}
Notice that $(\operatorname{curl}u\times\dot u)\cdot n=0$ on $\partial\Omega$ and
\begin{align*}
	\int_\Omega
	\sigma^m
	(\operatorname{curl}u\times\dot u)
	\cdot\nabla\operatorname{div}u\,dx
	-
	\int_\Omega
	\sigma^m
	u\cdot\nabla
	\operatorname{div}(\operatorname{curl}u\times\dot u)\,dx=
	\int_\Omega
	\sigma^m
	\operatorname{div}
	\left[
	(\operatorname{div}u)
	(\operatorname{curl}u\times\dot u)
	\right]dx=0.
\end{align*}
Substituting the above estimates on $J_2^{i}$ $(i=1, 2, 3)$ into \eqref{J2s}, we derive
\begin{align*}
	J_2
	=&
	-\mu\int_\Omega
	\sigma^m
	|\operatorname{curl}\dot u|^2dx
+
	\mu\int_\Omega
	\sigma^m
	\left(
	\operatorname{curl}u\times\nabla u^i
	\right)
	\cdot\partial_i\dot u\,dx
+
	\mu\int_\Omega
	\sigma^m
	\operatorname{curl}\dot u
	\cdot
	\left(
	\nabla u^i\times\partial_i u
	\right)dx
	\nonumber\\
	&-
	\mu\int_\Omega
	\sigma^m
	\operatorname{div}u\,
	\operatorname{curl}u
	\cdot
	\operatorname{curl}\dot u\,dx\notag\\
 \leq&
	-\mu\int_\Omega
	\sigma^m
	|\operatorname{curl}\dot u|^2dx
	+
	C\sigma^m
	\int_\Omega
	|\nabla u|^2
	|\nabla\dot u|dx
	\nonumber\\
	\leq&
	-\mu\int_\Omega
	\sigma^m
	|\operatorname{curl}\dot u|^2dx
	+
	C\sigma^m
	\|\nabla u\|_{L^4}^2
	\|\nabla\dot u\|_{L^2} \notag\\
	\leq&
	-\mu\int_\Omega
	\sigma^m
	|\operatorname{curl}\dot u|^2dx
	+
	\delta
	\sigma^m
	\|\nabla\dot u\|_{L^2}^2
	+
	C
	\sigma^m
	\|\nabla u\|_{L^4}^4,
\end{align*}
which combined with \eqref{d12},
\eqref{J0}, and \eqref{J1-final} yields that after
choosing $\delta$ small enough
\begin{align}\label{bq}
	&\frac{d}{dt}
	\left[
	\sigma^m
	\|\sqrt\rho\dot u\|_{L^2}^2
	+
	2\sigma^m\mathcal J(t)
	\right]
	+
	C\sigma^m
	\|\nabla\dot u\|_{L^2}^2
	\nonumber\\
	& \leq
	Cm\sigma^{m-1}\sigma^{\prime}
	\left(
	\|\sqrt\rho\dot u\|_{L^2}^2
	+
	\|\nabla u\|_{L^2}^2
	\right)
+
	C \sigma^m
	\big[
	(\|\nabla u\|_{L^2}^2+\|\nabla u\|_{L^2}^4)
	\|\sqrt\rho\dot u\|_{L^2}^2
	+
	\|\nabla u\|_{L^2}^2
	+
	\|\nabla u\|_{L^4}^4
	\big].
\end{align}

It is a position to prove the time-weighted estimate \eqref{mB1}. Take $m=1$ in \eqref{bq}, it follows from \eqref{f12} and \eqref{xxxx1} that
\begin{align}\label{m2}
	&\frac{d}{dt}
	\left[
	\sigma
	\|\sqrt\rho\dot u\|_{L^2}^2
	+
	2\sigma\mathcal J(t)
	\right]
	+
	C\sigma
	\|\nabla\dot u\|_{L^2}^2
	\nonumber\\
	&\leq
	C\sigma'
	\left(
	\|\sqrt\rho\dot u\|_{L^2}^2
	+
	\|\nabla u\|_{L^2}^2
	\right)
	+
	C\sigma
	\left[
	\left(
	\|\nabla u\|_{L^2}^2
	+
	\|\nabla u\|_{L^2}^4
	\right)
	\|\sqrt\rho\dot u\|_{L^2}^2
	+
	\|\nabla u\|_{L^2}^2
	+
	\|\nabla u\|_{L^4}^4
	\right]\nonumber\\
	&\leq
	C\sigma'
	\left(
	\|\sqrt\rho\dot u\|_{L^2}^2
	+
	\|\nabla u\|_{L^2}^2
	\right)
	+
	C\sigma
	\|\nabla u\|_{L^2}^2
	\|\sqrt\rho\dot u\|_{L^2}^2
	+
	C\sigma
	\|\sqrt\rho\dot u\|_{L^2}^4
	+
	C\sigma \left(\|\nabla u\|_{L^2}^2
	+
	\|P-\bar P\|_{L^2}^2 \right),
\end{align}
due to
\begin{align*}
	\|\nabla u\|_{L^4}^4
	&\leq
	C
	\|\sqrt\rho\dot u\|_{L^2}^3
	\left(
	\|\nabla u\|_{L^2}
	+
	\|P-\bar P\|_{L^2}
	\right)
	+
	C\|\nabla u\|_{L^2}^4
	+
	C\|P-\bar P\|_{L^4}^4\notag\\[2mm]
	&\leq
	C
	\|\sqrt\rho\dot u\|_{L^2}^4
	+
	C\left(\|\nabla u\|_{L^2}^2
	+
	\|P-\bar P\|_{L^2}^2 \right).
\end{align*}
Integrating \eqref{m2} over $[0,T]$, we find
\begin{align}
	&\sup_{0\leq t\leq T}\left( \sigma
	\|\sqrt\rho\dot u(t)\|_{L^2}^2
	+
	2\sigma  \mathcal J(t)\right)
	+
	C\int_0^T
	\sigma
	\|\nabla\dot u \|_{L^2}^2\,dt
	\nonumber\\
	&\leq
	C
	\int_0^{\sigma(T)}
	\left(
	\|\sqrt\rho\dot u \|_{L^2}^2
	+
	\|\nabla u \|_{L^2}^2
	\right)dt
	+
	C
	\int_0^T
	\sigma
	\|\nabla u \|_{L^2}^2
	\|\sqrt\rho\dot u \|_{L^2}^2\,dt
	\nonumber\\
	&\quad+
	C
	\int_0^T
	\sigma
	\|\sqrt\rho\dot u \|_{L^2}^4\,dt
	+
	C
	\int_0^T
	\sigma  \left(\|\nabla u \|_{L^2}^2
	+
	\| P-\bar P \|_{L^2}^2 \right)dt.
	\label{B1-pre}
\end{align}
Using Lemma \ref{Aclosed}, \eqref{basic1}, and
\eqref{xxxx1}, we have
\begin{align*}
	\sup_{0\leq t\leq T}
	\left( \sigma  \|\nabla u(t)\|_{L^2}^2\right)
	+
	\int_0^T
	\|\sqrt\rho\dot u \|_{L^2}^2\,dt
	+
	\int_0^T
	D(t)\,dt
	\leq
	CD(0).
\end{align*}
Therefore,
\begin{gather}\label{previous-estimates-B1}
	\int_0^T
	\sigma
	\|\nabla u\|_{L^2}^2
	\|\sqrt\rho\dot u \|_{L^2}^2\,dt
	 \leq
	\sup_{0\leq t\leq T}
	\big(
	\sigma \|\nabla u(t)\|_{L^2}^2
	\big)
	\int_0^T
	\|\sqrt\rho\dot u \|_{L^2}^2\,dt
 \leq
	CD(0)^2,\\ 	\label{U4-weight}
	\int_0^T
	\sigma
	\|\sqrt\rho\dot u\|_{L^2}^4\,dt
	\leq
	\sup_{0\leq t\leq T}
	\big(
	\sigma
	\|\sqrt\rho\dot u(t)\|_{L^2}^2
	\big)
	\int_0^T
	\|\sqrt\rho\dot u \|_{L^2}^2\,dt.
\end{gather}
It follows from \eqref{B1-pre}--\eqref{U4-weight} that
\begin{align}
	\sup_{0\leq t\leq T}
	\left( \sigma
	\|\sqrt\rho\dot u(t)\|_{L^2}^2\right)
	+
	C\int_0^T
	\sigma
	\|\nabla\dot u \|_{L^2}^2\,dt
	\leq
	CD(0)
	+
	CD(0)\sup_{0\leq t\leq T}
	\left( \sigma
	\|\sqrt\rho\dot u(t)\|_{L^2}^2\right).
	\label{B1-close}
\end{align}
Since $D(0)\leq2\varepsilon$, we may choose $\varepsilon>0$
sufficiently small such that
\begin{align}\label{xiao5}
CD(0)
\leq
2C\varepsilon
\leq
\frac12.
\end{align}
Thus the last term on the right-hand side of
\eqref{B1-close} can be absorbed, and we obtain
 \eqref{mB1}.

Finally, we show the exponential decay \eqref{w41}. Assume $T\geq1$. Since
$\sigma(t)=1$ for $t\geq1$, we take $m=0$ in \eqref{bq} and obtain
\begin{align}\label{B1-unweighted-long}
	\frac{d}{dt}
	\left[
	\|\sqrt\rho\dot u(t)\|_{L^2}^2
	+
	2\mathcal J(t)
	\right]
	+
	 \|\nabla\dot u\|_{L^2}^2
	\leq
	C
	\left[
	\left(
	\|\nabla u\|_{L^2}^2
	+
	\|\nabla u\|_{L^2}^4
	\right)
	\|\sqrt\rho\dot u\|_{L^2}^2
	+
	\|\nabla u\|_{L^2}^2
	+
	\|\nabla u\|_{L^4}^4
	\right].
\end{align}
By \eqref{mB1},
$
\|\sqrt\rho\dot u(t)\|_{L^2}^2
\leq
CD(0)$ for $t\geq1$, then one has
\begin{align*}
	\|\nabla u\|_{L^4}^4
	\leq
	C
	\|\sqrt\rho\dot u\|_{L^2}^2
	+
	CD(t),
	\qquad
	1\leq t\leq T.
\end{align*}
Hence,
\begin{align}\label{high-order-long}
	\frac{d}{dt}
\left(
	\|\sqrt\rho\dot u\|_{L^2}^2
	+
	2\mathcal J(t)
	\right)
	+
	 C\|\nabla\dot u\|_{L^2}^2
	\leq
	C
	\|\sqrt\rho\dot u\|_{L^2}^2
	+
	CD(t).
\end{align}
Define
\begin{align*}
	\mathcal C_3(t)
	\triangleq
	\|\sqrt\rho\dot u(t)\|_{L^2}^2
	+
	2\mathcal J(t)
	+
	C_3\mathcal C_2(t),
\end{align*}
where $C_3$ is a positive constant. By \eqref{Phi-equivalence} and \eqref{boundary3},  $C_3$ can be chosen sufficiently large so that
\begin{align}\label{C3-equivalence}
	\mathcal C_3(t)
	\sim
	\|\sqrt\rho\dot u(t)\|_{L^2}^2
	+
	D(t),
	\quad
	1\leq t\leq T.
\end{align}
Multiplying \eqref{cm2-decay-local} by $C_3$ and adding the
result to \eqref{high-order-long}, we obtain
\begin{align*}
	\frac{d}{dt}\mathcal C_3(t)
	+
	C
	\|\nabla\dot u\|_{L^2}^2
	+
	C
	\|\sqrt\rho\dot u\|_{L^2}^2
	+
	CD(t)
	\leq
	0,
\end{align*}
which together with \eqref{C3-equivalence} implies
\begin{align}\label{C3-decay}
	\frac{d}{dt}\mathcal C_3(t)
	+
	\eta_3\mathcal C_3(t)
	+
	C\|\nabla\dot u\|_{L^2}^2
	\leq
	0
\end{align}
for some sufficiently small $\eta_3>0$. Multiplying \eqref{C3-decay} by $e^{\eta_3t}$ and integrating the resultant over $[1,T]$ yields
\begin{align*}
	\sup_{1\leq t\leq T}\left( e^{\eta_3t}\mathcal C_3(t)\right)
	+
	 \int_1^T
	e^{\eta_3t}
	\|\nabla\dot u(t)\|_{L^2}^2\,dt
	\leq
	e^{\eta_3}\mathcal C_3(1)
\leq
CD(0),
\end{align*}
proving \eqref{w41}.
 \end{proof}

\subsection{Proof of Proposition \ref{pro}}

We are now in a position to close the bootstrap assumptions
\eqref{cp1}$_{2}$.
\begin{lemma}\label{sca}
Assume that \eqref{cp1} holds. Then
	\begin{align}\label{xue1}
		&\sup_{0\leq t\leq T}D(t)
		+\sup_{0\leq t\leq T}
		\|\nabla u(t)\|_{L^2}^2
		\int_0^T
		\|\nabla u(t)\|_{L^2}^2\,d t
		\leq
		\frac74\varepsilon.
	\end{align}
\end{lemma}

\begin{proof}
	Combining \eqref{basic1} with \eqref{xxxx1}, and using the smallness
	assumption $D(0)\leq\varepsilon_0<1$, we obtain
	\begin{align*}
		\sup_{0\leq t\leq T}D(t)
		+
		\sup_{0\leq t\leq T}\|\nabla u(t)\|_{L^2}^2
		\int_0^T\|\nabla u(t)\|_{L^2}^2\,dt
		\leq C\varepsilon_0 .
	\end{align*}
	Choosing
	\begin{align}\label{co}
		c_0
		\triangleq
		\max\left\{
		\frac12,\frac{4C}{7}
		\right\}.
	\end{align}
	and using $c_0\varepsilon_0\leq\varepsilon$, we have \eqref{xue1}.
\end{proof}

We now close the density part of the bootstrap assumption \eqref{cp1}. The argument is
based on the Zlotnik inequality. Since the compatibility condition is not assumed on the initial data, the higher-order estimates near the initial time require the
time-weighted  $\sigma(t)$. In contrast, the large-time estimates are
obtained from the previously established exponential decay. Therefore, we
treat these two time regimes separately.
\begin{lemma}\label{lemupper}
Assume that \eqref{cp1} holds. Then
	\begin{align}\label{zui}
		\sup_{0\leq t\leq T}
		\|\rho(t)\|_{L^\infty}
		\leq
		\frac74\rho_*.
	\end{align}
\end{lemma}

\begin{proof}
	Fix $(x,t)\in\Omega\times[0,T]$, and let $X(s;t,x)$ be the particle
	trajectory defined by
	\begin{align*}
	\frac{d}{ds}X(s;t,x)
	=
	u(X(s;t,x),s),
	\quad
	X(t;t,x)=x,
	\quad
	0\leq s\leq t.
\end{align*}
	Set
	\begin{align*}
	y(s)
	\triangleq
	\rho(X(s;t,x),s).
\end{align*}
	Then the continuity equation \eqref{a1}$_1$ gives
	\begin{align}\label{density-characteristic}
		y'(s)
		=
		-y(s)\operatorname{div}u(X(s;t,x),s).
	\end{align}
		Recall that $
	(2\mu+\lambda)\operatorname{div}u
	=F+P-\bar P$, we obtain
	\begin{align}\label{zlo}
	y'(s)
	=
	g(y(s))+b'(s),
\end{align}
	where
	\begin{align}\label{bbbb}
	g(\xi)
		\triangleq
		-\frac{\xi\bigl(P(\xi)-P(\bar\rho)\bigr)}
		{2\mu+\lambda},\ \
b(s)
		\triangleq
		-\frac{1}{2\mu+\lambda}
		\int_0^s
		y(\tau)
		\left[ F(X(\tau;t,x),\tau)
     -\bigl(\bar P(\tau)-P(\bar\rho)\bigr)\right] \,d\tau.
	\end{align}
	Since $P(\xi)=\xi^\gamma$ with $\gamma>1$,
	$
	\lim_{\xi\to\infty}g(\xi)=-\infty$.
	Moreover,
	$
	\bar\rho\leq\|\rho_0\|_{L^\infty}=\rho_*$,	and hence
	$
	g(\xi)\leq0$
	for
	$\xi\geq\rho_*$.
	Thus, applying Lemma \ref{zo} with $N_1=0$, we may take
	$\bar\zeta=\rho_*$.

	{\bf Step 1. Short-time estimate.}
	Let
$0\leq t_1\leq t_2\leq\sigma(T)\leq1$,
it follows from \eqref{bbbb} and
	$\rho\leq2\rho_*$ that
	\begin{align}\label{b-short-split}
		|b(t_2)-b(t_1)|
		\leq
		C\rho_*
		\int_{t_1}^{t_2}
		\|F(s)\|_{L^\infty}\,ds
		 +
		C\rho_*
		\int_{t_1}^{t_2}
		|\bar P(s)-P(\bar\rho)|\,ds.
	\end{align}
	We first estimate the pressure term. Since
	$
	\int_\Omega(\rho-\bar\rho)\,dx=0$,	we can write
	\begin{align*}
		\bar P -P(\bar\rho)
		=
		\frac1{|\Omega|}
		\int_\Omega
		\Big[
		P -P(\bar\rho)
		-
		P'(\bar\rho)(\rho-\bar\rho)
		\Big]\,dx\leq C \|G(\rho)\|_{L^{1}},
	\end{align*}
which combined with	\eqref{basic1} and $t_2-t_1\leq1$ leads to
	\begin{align}\label{barP-short-integral}
		\int_{t_1}^{t_2}
		|\bar P -P(\bar\rho)|\,dt
		\leq
		CD(0).
	\end{align}	
	We next estimate the effective viscous flux. Thanks to $\bar F=0$, the
	Gagliardo--Nirenberg inequality and Lemmas \ref{l24} and \ref{xu} give
	\begin{align}\label{F-infinity}
	\|F\|_{L^\infty}
	&\leq
	C\|F\|_{L^6}^{1/2}
	\|\nabla F\|_{L^6}^{1/2}\notag\\
	&\leq
	C\|\nabla F\|_{L^2}^{1/2}
	\|\nabla F\|_{L^6}^{1/2}\notag\\
	&\leq
		C\rho_*^{3/4}
		\left[
		\|\sqrt\rho \dot u\|_{L^2}^{1/2}
		\|\nabla\dot u\|_{L^2}^{1/2}
		+
		\|\sqrt\rho \dot u\|_{L^2}^{1/2}
		\|\nabla u\|_{L^2}
		\right].
	\end{align}	
	On $[0,\sigma(T)]$, one has $\sigma(t)=t$. Hence, by H\"older's
	inequality,
	\begin{align}
	 \int_0^{\sigma(T)}
		\|\sqrt\rho \dot u \|_{L^2}^{1/2}
		\|\nabla\dot u \|_{L^2}^{1/2}\,dt
		&=
		\int_0^{\sigma(T)}
		\bigl(
		\|\sqrt\rho \dot u \|_{L^2}^2
		\bigr)^{1/4}
		\bigl(
		t\|\nabla\dot u \|_{L^2}^2
		\bigr)^{1/4}
		t^{-1/4}\,dt
		\nonumber\\
		&\leq
		C
		\left(
		\int_0^T
		\|\sqrt\rho \dot u \|_{L^2}^2\,dt
		\right)^{1/4}
		\left(
		\int_0^T
		\sigma
		\|\nabla\dot u \|_{L^2}^2\,dt
		\right)^{1/4}
		\nonumber\\[2mm]
		&\leq
		CD(0)^{1/2},
		\label{density-short-F1}
	\end{align}
	where we have used \eqref{xxxx1} and Lemma \ref{lemb}.
	Similarly,
	\begin{align}
	 \int_0^{\sigma(T)}
		\|\sqrt\rho \dot u \|_{L^2}^{1/2}
		\|\nabla u \|_{L^2}\,dt\leq
		C
		\left(
		\int_0^T
		\|\sqrt\rho \dot u \|_{L^2}^2\,dt
		\right)^{1/4}
		\left(
		\int_0^T
		\|\nabla u \|_{L^2}^2\,dt
		\right)^{1/2}
		 \leq
		CD(0)^{3/4}.
		\label{density-short-F2}
	\end{align}
	It follows from \eqref{F-infinity}--\eqref{density-short-F2} that
	\begin{align*}
		\int_0^{\sigma(T)}
		\|F \|_{L^\infty}\,dt
		\leq
		C\Big(
		D(0)^{1/2}
		+
		D(0)^{3/4}
		\Big),
	\end{align*}
which combined with \eqref{b-short-split} and \eqref{barP-short-integral} yields
	\begin{align}\label{xaio1}
 b(t_2)-b(t_1)
	\leq
	C\rho_*
	\left(
	D(0)^{1/2}
	+
	D(0)^{3/4}
	+
	D(0)
	\right)\leq C\rho_*D(0)^{1/2}\leq \frac12\rho_*
	\end{align}
provided $\varepsilon$ in \eqref{cp1} is
chosen sufficiently small.
	
	Applying Lemma \ref{zo} with
$N_0=\frac12\rho_*$, $N_1=0$, and $\bar\zeta=\rho_*$.
	Since
	$
	y(0)
	=
	\rho_0(X(0;t,x))
	\leq\rho_*$, Zlotnik's inequality implies
	\begin{align*}
	y(t)
	\leq
	\max\{y(0),\rho_*\}
	+
	\frac12\rho_*
	\leq
	\frac32\rho_*.
\end{align*}
	Consequently,
	\begin{align}\label{dresult}
		\sup_{0\leq t\leq\sigma(T)}
		\|\rho(t)\|_{L^\infty}
		\leq
		\frac32\rho_*.
	\end{align}
	
	{\bf Step 2. Large-time estimate.}
	Let
$1\leq t_1\leq t_2\leq T$ and choose $0<\eta_4<\min\{\eta_2,\eta_3\}$.
	Then
	\begin{align}\label{d1}
	 \sup_{1\leq t\leq T}
		\big[e^{\eta_4t}
		\left(
		D(t)
		+
		\|\sqrt\rho \dot u(t)\|_{L^2}^2
		\right)\big]
	    +
		\int_1^T
		e^{\eta_4t}
		\left(
		\|\sqrt\rho \dot u \|_{L^2}^2
		+
		\|\nabla\dot u \|_{L^2}^2
		\right)\,dt
		\leq
		CD(0).
	\end{align}
	In particular,
	\begin{align}\label{d2}
		\|\nabla u(t)\|_{L^2}
	    \leq
		CD(0)^{1/2}e^{-\eta_4t/2},\quad
		\|\sqrt\rho \dot u(t)\|_{L^2}^{1/2}
		 \leq
		CD(0)^{1/4}e^{-\eta_4t/4}.
	\end{align}
	Using \eqref{d1} and \eqref{d2}, we have
	\begin{align*}
	&\int_{t_1}^{t_2}
		\|\sqrt\rho \dot u \|_{L^2}^{1/2}
		\|\nabla\dot u \|_{L^2}^{1/2}\,dt\notag\\
		&=
		\int_{t_1}^{t_2}
		\left(
		e^{\eta_4t}
		\|\sqrt\rho \dot u \|_{L^2}^2
		\right)^{1/4}
		\left(
		e^{\eta_4t}
		\|\nabla\dot u \|_{L^2}^2
		\right)^{1/4}
		e^{-\eta_4t/2}\,dt
		\notag\\
		&\leq C
		\left(
		\sup_{t_1\leq t\leq t_2}
		e^{\eta_4t}
		\|\sqrt\rho \dot u(t)\|_{L^2}^2
		\right)^{1/4}
		\int_{t_1}^{t_2}
		\left(
		e^{\eta_4t}
		\|\nabla\dot u \|_{L^2}^2
		\right)^{1/4}
		e^{-\eta_4t/2}\,dt\notag\\
 &\leq C
 \left(
 \sup_{t_1\leq t\leq t_2}
 e^{\eta_4t}
 \|\sqrt\rho \dot u(t)\|_{L^2}^2
 \right)^{1/4}
		\left(
		\int_{t_1}^{t_2}
		e^{\eta_4t}
		\|\nabla\dot u \|_{L^2}^2\,dt
		\right)^{1/4}
		\left(
		\int_{t_1}^{t_2}
		e^{-2\eta_4t/3}\,dt
		\right)^{3/4}
		\notag\\[2mm] 
		&\leq
			C D(0)^{1/2}
		e^{-\eta_4t_1/2}
	\end{align*}
and
	\begin{align*}
	 \int_{t_1}^{t_2}
		\|\sqrt\rho  \dot u \|_{L^2}^{1/2}
		\|\nabla u \|_{L^2}\,dt \leq
		C D(0)^{3/4}
		\int_{t_1}^{t_2}
		e^{-3\eta_4t/4}\,dt
		 \leq
		C D(0)^{3/4}
		e^{-3\eta_4t_1/4}.
	\end{align*}
Hence
	\begin{align}\label{density-large-F-final}
		\int_{t_1}^{t_2}
		\|F \|_{L^\infty}\,dt
		\leq
		C
		\left[
		D(0)^{1/2}e^{-\eta_4t_1/2}
		+
		D(0)^{3/4}e^{-3\eta_4t_1/4}
		\right].
	\end{align}

	As in \eqref{b-short-split}, we obtain
	\begin{align}\label{xiao2}
	 b(t_2)-b(t_1)
		\leq
		C\rho_*
		\left[
		D(0)^{1/2}e^{-\eta_4t_1/2}
		+
		D(0)^{3/4}e^{-3\eta_4t_1/4}
		+
		D(0)
		\right]\leq C\rho_*D(0)^{1/2}\leq \frac14\rho_*,
	\end{align}
	where the smallness of $\varepsilon$ in \eqref{cp1} has been used.
	Thus, on $[1,T]$, applying Lemma \ref{zo} with
	$N_0=\frac14\rho_*$, $N_1=0$, and $\bar\zeta=\rho_*$, we have
	\begin{align*}
	y(t)
	\leq
	\max\{y(1),\rho_*\}
	+
	\frac14\rho_*
	\leq
	\frac74\rho_*,
\end{align*}
which gives
\begin{align*}
	\sup_{1\leq t\leq T}
	\|\rho(t)\|_{L^\infty}
	\leq
	\frac74\rho_*. \tag*{\qedhere}
\end{align*}
\end{proof}

\section{A \textit{priori} estimates (II): higher-order estimates}\label{se4}


Before proceeding to the higher-order estimates, we summarize the
lower-order bounds obtained in Section \ref{se3}.
\begin{corollary}\label{wen}
	Assume that  \eqref{sca1} holds. Then there
	exist constants $\eta>0$ and $C$, such that
	\begin{align}\label{wen1}
	 \sup_{0\leq t\leq T}
		\left(
		\|\nabla u(t)\|_{L^2}^2
		+
		\|(P-\bar P)(t)\|_{L^2}^2
		\right)
		+
		\int_0^T
		\left(
		\|\nabla u \|_{L^2}^2
		+
		\|  P-\bar P \|_{L^2}^2
		+
		\|\sqrt\rho\dot u \|_{L^2}^2
		\right)\,dt
		\leq
		C,
	\end{align}
	and, if $T\geq1$,
	\begin{align}\label{wen2}
	\sup_{1\leq t\leq T}
		\left[ e^{\eta t}
		\left(
		\|\sqrt\rho \dot u(t)\|_{L^2}^2
		+
		\|\nabla u(t)\|_{L^2}^2
		\right)\right]
		+
		\int_1^T
		e^{\eta t}
		\left(
		\|\sqrt\rho \dot u \|_{L^2}^2
		+
		\|\nabla\dot u \|_{L^2}^2
		\right)\,dt
		\leq
		C.
	\end{align}
	Here $\eta$  be chosen as
	$
	0<\eta<\min\{\eta_2,\eta_3\}$,
	where $\eta_2$ and $\eta_3$ are the constants appearing in
	\eqref{fir} and \eqref{w41}.
\end{corollary}

We next establish the higher-order estimates needed to propagate the
regularity of the density and to recover the spatial regularity of the
velocity for positive time.
\begin{lemma}\label{lem-high-order}
	Assume that \eqref{sca1} holds. For  $q\in(3,6)$, there
	exists a positive constant $C$ depending on $T$ such that
	\begin{gather}\label{higher-1}
		\sup_{0\leq t\leq T}
		\left(
		t\|\nabla^2u(t)\|_{L^2}^2
		\right)
		+
		\int_0^T
		\left( t\|\nabla u_t \|_{L^2}^2+t\|\nabla^{2} u \|_{L^{q}}^2\right) \,dt
		\leq C,\\
\label{higher-2}
	 \sup_{0\leq t\leq T}
		\left(
		\|\nabla\rho(t)\|_{ L^q}
		+\|\rho_t(t)\|_{L^2}^2
		\right)
		 +
		\int_0^T
		\left(
		\|\nabla^2u \|_{ L^q}
		+
		\|\nabla u \|_{L^\infty}
		\right)dt
		\leq C.
	\end{gather}
\end{lemma}
\begin{proof}
We first prove that
	\begin{align}\label{rho-dotu-L1q}
		\int_0^T
		\|\rho\dot u \|_{L^q}\,dt
		\leq C.
	\end{align}
It follows from the interpolation inequality and \eqref{d11} that
	\begin{align}\label{0000}
		\|\rho\dot u\|_{L^q}
		 \leq
		C
		\|\rho\dot u\|_{L^2}^{1-\theta}
		\|\rho\dot u\|_{L^6}^{\theta}
	 \leq
		C
		\|\sqrt\rho \dot u\|_{L^2}^{1-\theta}
		\|\dot u\|_{L^6}^{\theta}
	 \leq
		C
		\|\sqrt\rho \dot u\|_{L^2}^{1-\theta}
		\|\nabla\dot u\|_{L^2}^{\theta}
		+
		C
		\|\sqrt\rho \dot u\|_{L^2}^{1-\theta}
		\|\nabla u\|_{L^2}^{2\theta},
	\end{align}
where $\theta\triangleq\frac{3(q-2)}{2q}\in\left( \frac12, 1\right)$.
We estimate the two terms on the right-hand side of \eqref{0000} separately.
	
If $T<1$, from \eqref{xxxx1}, \eqref{mB1}, and H\"older's inequality, one has that
	\begin{align*}
		&\int_0^{\sigma(T)}
		\|\sqrt\rho \dot u \|_{L^2}^{1-\theta}
		\|\nabla\dot u \|_{L^2}^{\theta}\,dt
		\nonumber\\
		&=
		\int_0^{\sigma(T)}
		\bigg(
		\|\sqrt\rho \dot u \|_{L^2}^2
		\bigg)^{\frac{1-\theta}{2}}
		\bigg(
		t\|\nabla\dot u \|_{L^2}^2
		\bigg)^{\frac{\theta}{2}}
		t^{-\frac{\theta}{2}}\,dt
		\nonumber\\
		&\leq C
		\bigg(
		\int_0^{\sigma(T)}
		\|\sqrt\rho \dot u \|_{L^2}^2\,dt
		\bigg)^{\frac{1-\theta}{2}}
		\bigg(
		\int_0^{\sigma(T)}
		t\|\nabla\dot u \|_{L^2}^2\,dt
		\bigg)^{\frac{\theta}{2}}
		\bigg(
		\int_0^{\sigma(T)}
		t^{-\theta}\,dt
		\bigg)^{\frac12}\leq C.
	\end{align*}	
	Similarly,
	\begin{align*}
		\int_0^{\sigma(T)}
		\|\sqrt\rho \dot u \|_{L^2}^{1-\theta}
		\|\nabla u \|_{L^2}^{2\theta}\,dt
		&\leq C
		\bigg(
		\sup_{0\leq t\leq T}
		\|\nabla u(t)\|_{L^2}^2
		\bigg)^\theta
		\int_0^{\sigma(T)}
		\|\sqrt\rho \dot u \|_{L^2}^{1-\theta}\,dt
		\nonumber\\
		&\leq
		C
		\bigg(
		\sup_{0\leq t\leq T}
		\|\nabla u(t)\|_{L^2}^2
		\bigg)^\theta
		\left(
		\int_0^T
		\|\sqrt\rho \dot u \|_{L^2}^2\,dt
		\right)^{\frac{1-\theta}{2}}\leq C.
	\end{align*}
If $T\geq1$, one deduces from \eqref{wen2} and H\"older's inequality that
	\begin{align*}
		&\int_1^T
		\|\sqrt\rho \dot u \|_{L^2}^{1-\theta}
		\|\nabla\dot u \|_{L^2}^{\theta}\,dt
		\nonumber\\
		&=
		\int_1^T
		\left(
		e^{\eta t}
		\|\sqrt\rho \dot u \|_{L^2}^2
		\right)^{\frac{1-\theta}{2}}
		\left(
		e^{\eta t}
		\|\nabla\dot u \|_{L^2}^2
		\right)^{\frac{\theta}{2}}
		e^{-\frac{\eta t}{2}}\,dt
		\nonumber\\
		&\leq
		C\left(
		\int_1^T
		e^{\eta t}
		\|\sqrt\rho \dot u \|_{L^2}^2\,dt
		\right)^{\frac{1-\theta}{2}}
		\left(
		\int_1^T
		e^{\eta t}
		\|\nabla\dot u \|_{L^2}^2\,dt
		\right)^{\frac{\theta}{2}}
		\left(
		\int_1^\infty
		e^{-\eta t}\,dt
		\right)^{1/2}\leq C.
	\end{align*}
Moreover, the pointwise exponential decay in \eqref{wen2} gives
\begin{align*}
	\int_1^T
	\|\sqrt\rho \dot u \|_{L^2}^{1-\theta}
	\|\nabla u \|_{L^2}^{2\theta}\,dt
	\leq
	C\int_1^T
	e^{-\frac{1+\theta}{2}\eta t}\,dt
	\leq C.
\end{align*}
From the above four estimates, one sees that \eqref{rho-dotu-L1q} holds.

We next propagate the $W^{1,q}$-regularity of the density. Differentiating \eqref{a1}$_{1}$ with respect to $x$ yields
	\begin{align}\label{gradient-rho-equation}
		\partial_t\nabla\rho
		+
		u\cdot\nabla(\nabla\rho)
		+
		(\nabla u)^{\mathrm T}\nabla\rho
		+
		(\divv u)\nabla\rho
		+
		\rho\nabla\divv u
		=
		0.
	\end{align}
	Testing \eqref{gradient-rho-equation} by
	$q|\nabla\rho|^{q-2}\nabla\rho$ and integrating the resultant over $\Omega$
	leads to
	\begin{align}\label{grad-rho-q-1}
		\frac{d}{dt}
		\|\nabla\rho(t)\|_{L^q}
		\leq
		C
		\|\nabla u\|_{L^\infty}
		\|\nabla\rho\|_{L^q}
		+
		C
		\|\nabla^2u\|_{L^q}.
	\end{align}	
	To estimate $\|\nabla^2u\|_{L^q}$, we rewrite the momentum equation  \eqref{a1}$_{2}$ as
	\begin{align*}
		-\mu\Delta u
		-
		(\lambda+\mu)\nabla\divv u
		=
		-\rho\dot u-\nabla P.
	\end{align*}
By $W^{2, q}$-estimate of the Lam\'{e} system (see \cite[Lemma 3.1]{HWW12}), we obtain that
	\begin{align}
		\|\nabla^2u\|_{L^q}
		\leq
		C
		\left(
		\|\rho\dot u\|_{L^q}
		+
		\|\nabla P\|_{L^q}
		+
		\|u\|_{L^q}
		\right)
		\leq
		C
		\left(
		\|\rho\dot u\|_{L^q}
		+
		\|\nabla\rho\|_{L^q}
		+
		\|\nabla u\|_{L^2}
		\right).
		\label{H}
	\end{align}
One needs to control $\|\nabla u\|_{L^\infty}$ in \eqref{grad-rho-q-1}.
Since $\bar F=0$, one infers from Gagliardo--Nirenberg inequality,
Poincar\'e's inequality, and \eqref{ff8} that
	\begin{align}\label{F}
		\|F\|_{L^\infty}
	 \leq
		C
		\|F\|_{L^2}^{1-\alpha_q}
		\|\nabla F\|_{L^q}^{\alpha_q}
		 \leq
		C
		\|\rho\dot u\|_{L^q}^{\alpha_q}
	\end{align}
with $\alpha_q=\frac{3q}{5q-6}\in(0,1)$.
Owing to $(\operatorname{curl}u\times n)\big|_{\partial\Omega}=0$,
	Gagliardo--Nirenberg inequality together with \eqref{f8} gives
	\begin{align}\label{W}
		\|\operatorname{curl}u\|_{L^\infty}
		\leq
		C
		\left(
		1+
		\|\rho\dot u\|_{L^q}^{\alpha_q}
		\right).
	\end{align}
	Combining \eqref{H}--\eqref{W} with Lemma
	\ref{lem-bkm}, one finds that
	\begin{align}\label{BKM}
		\|\nabla u\|_{L^\infty}
		\leq
		C
		\left(
		1+
		\|\rho\dot u\|_{L^q}^{\alpha_q}
		\right)
		\ln
		\left(
		e+
		\|\nabla\rho\|_{L^q}
		+
		\|\rho\dot u\|_{L^q}
		\right)
		+
		C\leq
		C
		\left(
		1+\|\rho\dot u\|_{L^q}
		\right)
		\ln
		\left(
		e+\|\nabla\rho\|_{L^q}
		\right).
	\end{align}	
	Substituting \eqref{H} and \eqref{BKM} into
	\eqref{grad-rho-q-1} implies
	\begin{align*}
		\frac{d}{dt}
		\ln
		\left(
		e+\|\nabla\rho(t)\|_{L^q}
		\right)
		\leq
		C
		\left(
		1+\|\rho\dot u\|_{L^q}
		\right)
		\ln
		\left(
		e+\|\nabla\rho\|_{L^q}
		\right),
	\end{align*}
which along with Gronwall's inequality and \eqref{rho-dotu-L1q} shows that
	\begin{align}\label{grad-rho-q-final}
		\sup_{0\leq t\leq T}
		\|\nabla\rho(t)\|_{L^q}
		\leq C(T).
	\end{align}
	Inserting \eqref{grad-rho-q-final} into \eqref{BKM} and using
	\eqref{rho-dotu-L1q}, we further deduce
	\begin{align}\label{grad-u-infty-integral}
		\int_0^T
		\|\nabla u \|_{L^\infty}\,dt
		\leq C.
	\end{align}

We can now recover the spatial regularity of $u$. From \eqref{wen1}, \eqref{rho-dotu-L1q}, \eqref{H}, and
	\eqref{grad-rho-q-final}, we obtain
	\begin{align}\label{grad2-u-q-integral}
		\int_0^T
	  \|\nabla^2u \|_{L^q}dt
		\leq C.
	\end{align}
The continuity equation \eqref{a1}$_{1}$ gives
	\begin{align*}
		\|\rho_t\|_{L^2}
		&\leq
		\|\nabla u\|_{L^2}
		\|\nabla\rho\|_{L^3}
		+
		C\|\nabla u\|_{L^2}.
	\end{align*}
It follows from the interpolation inequality and 
\eqref{grad-rho-q-final} that
\begin{align}\label{rho-t-final}
	\sup_{0\leq t\leq T}
	\|\nabla\rho(t)\|_{L^3}
	\leq C\quad \text{and} \quad
		\sup_{0\leq t\leq T}
		\|\rho_t(t)\|_{L^2}
		\leq C.
	\end{align}
Moreover, \eqref{mB1}, \eqref{wen1}, \eqref{higher-2},
and \eqref{rho-t-final} imply that
\begin{align}\label{zui1}
	\sup_{0\leq t\leq T}
	\big(t\|\nabla^2u(t)\|_{L^2}^2\big)
	\leq C.
\end{align}	
	By
	$
	u_t
	=
	\dot u-u\cdot\nabla u$ and \eqref{f12},
	we have
	\begin{align*}
		\int_0^T
		t\|\nabla u_t \|_{L^2}^2\,dt
		\leq
		C
		\int_0^T
		t\|\nabla\dot u \|_{L^2}^2\,dt
	 +
		C
		\int_0^T
		t
		\left(
		\|\nabla^2u \|_{L^2}^3
		+
		1
		\right)dt.
	\end{align*}
	The first term can be bounded by \eqref{mB1}
	due to $t\leq C(T)\sigma(t)$. For the cubic term,
	\eqref{grad2-u-q-integral} and \eqref{zui1} yield
	\begin{align*}
		\int_0^T
		t\|\nabla^2u \|_{L^2}^3\,dt
		\leq
		\sup_{0\leq t\leq T}
		\big(\sqrt{t}\|\nabla^2u(t)\|_{L^2}
		\big)
		\int_0^T
		\sqrt{t}\|\nabla^2u \|_{L^2}^2\,dt
		\leq
		C
		\int_0^T
		\|\nabla^2u \|_{L^2}^2\,dt
		\leq C.
	\end{align*}
	Therefore,
	\begin{align}\label{ut-gradient-final}
		\int_0^T
		t\|\nabla u_t \|_{L^2}^2\,dt
		\leq C.
	\end{align}
	
Finally, we establish the time-weighted $W^{2,q}$-estimate for the velocity. Multiplying \eqref{0000} by $\sigma$ and integrating the resultant over $[0,T]$, we obtain
	\begin{align*}
		\int_0^T
		\sigma
		\|\rho\dot u \|_{L^q}^2\,dt&\leq C\int_0^T
		\sigma
		\|\sqrt\rho \dot u \|_{L^2}^{2(1-\theta)}
		\|\nabla\dot u \|_{L^2}^{2\theta}\,dt+C\int_0^T
		\sigma
		\|\sqrt\rho \dot u \|_{L^2}^{2(1-\theta)}
		\|\nabla u \|_{L^2}^{4\theta}\,dt\\
		&\leq
		C
		\bigg(
		\sup_{0\leq t\leq T}
		\sigma \|\sqrt\rho \dot u(t)\|_{L^2}^2
		\bigg)^{1-\theta}
		T^{1-\theta}
		\left(
		\int_0^T
		\sigma\|\nabla\dot u \|_{L^2}^2\,dt
		\right)^\theta\\
		&\quad+C
		\bigg(
		\sup_{0\leq t\leq T}
		\sigma \|\sqrt\rho \dot u(t)\|_{L^2}^2
		\bigg)^{1-\theta}
		\bigg(
		\sup_{0\leq t\leq T}
		\|\nabla u(t)\|_{L^2}^2
		\bigg)^{2\theta}
		\int_0^T\sigma^\theta\,dt\leq C,
	\end{align*}
which together with \eqref{wen1} and \eqref{grad-rho-q-final} leads to
	\begin{align*}
		\int_0^T
		t\|\nabla^{2} u \|_{L^{q}}^2\,dt
		\leq
		C\int_0^T
		t
		\left(
		\|\rho\dot u \|_{L^q}^2
		+
		\|\nabla\rho \|_{L^q}^2
		+
		\|\nabla u \|_{L^2}^2
		\right)dt
		\leq
		C. \tag*{\qedhere}
	\end{align*} 
\end{proof}

\section{Proof of Theorem \ref{th1}}\label{se5}

With the \textit{a priori} estimates obtained in Sections \ref{se3} and \ref{se4}, we now complete the proof of Theorem \ref{th1}.

\begin{proof}[Proof of Theorem \ref{th1}]
	By Proposition \ref{local}, there exists $T_*>0$ such that
	\eqref{a1}--\eqref{a4} admits a unique local strong solution
	$(\rho,u)$ on $\Omega\times(0,T_*]$. The assumptions on the initial data ensure that the bootstrap
	conditions \eqref{cp1} are satisfied at $t=0$. By the time
	continuity of the local strong solution, there exists
	$T_1\in(0,T_*]$ such that \eqref{cp1} remains valid on
	$[0,T_1]$.
	
	Define
	\begin{align*}
		T^*
		\triangleq
		\sup
		\Big\{
		T>0:\,
		(\rho,u)\ \text{is a strong solution to }
		\eqref{a1}-\eqref{a4}
		\text{ on }\Omega\times[0,T]~
	\text{and satisfies }\eqref{cp1}\text{ on }[0,T]
		\Big\}.
	\end{align*}
	Clearly, $T^*\geq T_1>0$. We claim that $T^*=\infty$. To this end, suppose otherwise that
	$T^*<\infty$. For any $0<\tau<T<T^*$, \eqref{wen1} and \eqref{higher-1}  yield that
	\begin{align}\label{ccon1}
		u\in C\big([\tau,T];H^1\big).
	\end{align}
	Moreover, by \eqref{cp1}, \eqref{basic1}, \eqref{xxxx1}, \eqref{higher-1}, \eqref{higher-2}, and Sobolev's inequality, we
also obtain
\begin{align}\label{ccon2}
	\rho u\in C\left([0,T];L^{2}\right).
\end{align}
	In addition, the classical continuity result
	\cite[Chapter 3, Lemma 1.4]{Tem24} together with \eqref{wen1}--\eqref{higher-2} implies
	\begin{align}\label{on2}
		\rho
		\in
		C\big([0,T];W^{1,q}\big).
	\end{align}
	
	We now return to the maximal time $T^*$.
	Passing to the limit as $t\rightarrow T^{*}$ and using
	\eqref{ccon1}--\eqref{on2}, we conclude that
	\begin{align*}
		(\rho,u)(T^*)
		\triangleq
		\lim_{t\rightarrow T^*}(\rho,u)(t)
	\end{align*}
	exists  and satisfies
	the hypotheses required by Proposition \ref{local}.
	Taking $(\rho,u)(T^*)$ as the new initial data, Proposition
	\ref{local} extends the strong solution to
	$[0,T^*+\delta]$ for some $\delta>0$.
	This contradicts the definition of $T^*$.
	Therefore $T^*=\infty$, and the local strong solution extends
	globally in time.
	Finally, with \eqref{w1}, \eqref{fir}, and \eqref{w41}  at hand, \eqref{decay} holds. \eqref{blow} can be obtained by similar arguments as those in \cite{CL23}, and we omit the details for simplicity.
\end{proof}

\section*{Conflict of interest}
The authors have no conflicts to disclose.

\section*{Data availability}
No data was used for the research described in the article.

\end{document}